\documentclass[11pt]{article}
\usepackage{amsmath,amssymb,tikz-cd}
\usepackage{mathrsfs}
\usepackage{mathrsfs}
\usepackage{mathrsfs}
\usepackage{mathtools}
\usepackage{amsfonts}
\usepackage{amssymb}
\usepackage{amsmath}
\usepackage{amsthm}
\usepackage{color}
\usepackage{enumitem}
\usepackage{bm}
\theoremstyle{plain}
\usepackage{amsthm}
\usepackage{url}

\theoremstyle{plain}
\newtheorem{theorem}{Theorem}[section]
\newtheorem{lemma}[theorem]{Lemma}
\newtheorem{proposition}[theorem]{Proposition}
\numberwithin{equation}{section}
\newtheorem{cor}[theorem]{Corollary}

\newcommand{\e}{\mathrm{e}}
\newcommand{\R}{\mathbb{R}}

\newcommand{\N}{\mathbb{N}}
\newcommand{\Z}{\mathbb{Z}}
\newcommand{\PP}{\mathcal{P}}

\newcommand{\supp}{\operatorname{supp}}
\begin{document}
		\title{{\Large\bf Rudin's Problem on Uniformly Bounded Homogeneous Orthonormal Bases}}
		\author{{\normalsize Penghui Wang, Fan Zhang and Zeyou Zhu}}
	\date{}
	\maketitle
	\begin{minipage}{14cm} {\bf Abstract.}
		Rudin's problem asks whether, in the unit ball of $\mathbb C^d$ with $d\ge2$ fixed, the Hardy space $H^2(\partial\mathbb B_d)$ admits a uniformly bounded orthonormal basis consisting of  homogeneous holomorphic polynomials. For $d=2$ and $d=3$ the answer is due to Bourgain, but the general-dimensional case has remained open. This paper gives an affirmative answer to Rudin's problem in every fixed dimension. More precisely, we prove that there is an orthonormal basis consisting of homogeneous polynomials $\{p_j\}_{j=1}^\infty$ of $H^2(\partial\mathbb B_d)$ such that 
		$$
		\|p_j\|_{L^\infty(\partial\mathbb B_d)}\leq \exp\left(\frac12d\log d+O(d)\right).
		$$
		Here the implicit constant in $O(d)$ is independent of $d$ and $j$.
		
		
	\end{minipage}
	\maketitle

	\bigskip

	\section{Introduction}
			Let $\mathbb B_d\subseteq\mathbb C^d$ be the open unit ball,
			and let $\sigma_d$ be the normalized surface measure on $\partial\mathbb B_d$.
			The Hardy space $H^2(\partial\mathbb B_d)$ is one of the most
			important function spaces over the unit ball, which is
			defined by
			\[
			H^2(\partial\mathbb B_d)
			=
			\left\{
			f:\mathbb B_d\to\mathbb C:
			f \text{ is holomorphic on } \mathbb B_d
			\text{ and }
			\sup_{0<r<1}
			\int_{\partial\mathbb B_d}
			|f(r\zeta)|^2\,d\sigma_d(\zeta)<\infty
			\right\}.
			\]
			We identify each such function with its radial boundary values.
			It is well-known that $H^2(\partial\mathbb B_d)$ is a
			reproducing kernel Hilbert space of holomorphic functions with the reproducing kernel
			\[
			K(z,w)
			=
			\frac{1}{(1-\langle z,w\rangle)^d}, \quad  z,w\in\mathbb B_d.
			\]
			
			As a cornerstone of operator theory on function space,
			Beurling's theorem \cite{Beurling1949} characterizes
			the invariant subspaces
			of the shift operator by inner functions in the
			Hardy space over the unit disc.
			Motivated by Beurling's theorem, Rudin \cite{Rudin1980} asked
			whether $H^2(\partial\mathbb B_d)$ admits a nonconstant inner function. 
			To study the existence and the structure of inner functions,
			Rudin \cite{Rudin1986} proposed the problem of whether
			$H^2(\partial\mathbb B_d)$ admits a uniformly bounded
			orthonormal basis consisting of homogeneous holomorphic
			polynomials.
			Although the existence problem of inner functions in
			$H^2(\partial\mathbb B_d)$ was resolved by
			\cite{Aleksandrov1983,Low1982},
			the existence of a uniformly bounded orthonormal basis consisting of homogeneous holomorphic polynomials had remained open in general dimensions.
			
			For $d=2,3$, Rudin's problem was resolved by Bourgain
			in \cite{Bourgain1985,Bourgain2016}, respectively.
			The two-dimensional construction uses the
			Golay--Rudin--Shapiro signs, which can also be replaced
			by suitable quadratic phases using estimates of Hardy
			and Littlewood \cite{HardyLittlewood1914}.
			The three-dimensional construction uses a quadratic phase,
			a rearrangement of the triangular index set,
			and techniques based on the Haar system.
			
			In arbitrary dimension,
			a related study due to Ryll and Wojtaszczyk
			\cite{RyllWojtaszczyk1983} showed
			that $\mathcal P_N^{(d)}$ contains a polynomial $p$
			satisfying
			\[
			\|p\|_\infty
			=
			\sup_{z\in\partial\mathbb B_d}|p(z)|
			=1
			\quad\text{and}\quad
			\|p\|_{H^2(\partial\mathbb B_d)}\geq C_d,
			\]
			where $C_d>0$ is independent of $N$ and
			\[
			\mathcal P_N^{(d)}
			=
			\operatorname{span}
			\left\{
			z^\alpha:
			\alpha\in\mathbb N_0^d,\ |\alpha|=N
			\right\},\quad
			\mathbb N_0^d
			=\{(\alpha_1,\ldots,\alpha_d):\alpha_j\in\{0,1,2,\ldots\},
			\ 1\le j\le d\}.
			\]
			Obviously,
			\[
			D_{N,d}
			=
			\dim\mathcal P_N^{(d)}
			=
			\binom{N+d-1}{d-1}.
			\]
			 Shiffman \cite{Shiffman2015}
			constructed uniformly bounded orthonormal systems
			in $\mathcal P_N^{(d)}$ with at least
			$\beta_dD_{N,d}$ elements, where $\beta_d>0$.
			Marzo and Ortega-Cerd\`a
			\cite{MarzoOrtegaCerda2015} obtained such systems
			with at least $(1-\varepsilon)D_{N,d}$ elements
			for every $\varepsilon>0$, with a bound depending
			on $d$ and $\varepsilon$.
			Without the homogeneity requirement, uniformly bounded
			polynomial orthonormal bases are already known from
			the constructions of Ovsepian--Pe{\l}czy{\'n}ski
			\cite{OvsepianPelczynski1975} and Olevskii
			\cite{Olevskii1975}, together with the result of
			Ryll and Wojtaszczyk; see also
			\cite[Appendix~I]{Rudin1986}.
			The homogeneous problem requires a complete
			orthonormal basis in each $\mathcal P_N^{(d)}$,
			with a common bound independent of $N$.
			
			Rudin's problem can be naturally considered in the framework of unitarily invariant reproducing kernel Hilbert spaces. The theory of reproducing kernel Hilbert spaces was developed
			by Aronszajn \cite{Aronszajn1950}.
			On the unit ball, the unitarily invariant spaces include
			the Hardy and weighted Bergman spaces, as well as the
			Drury--Arveson space, which plays a central role in
			multivariable operator theory \cite{Drury1978,Arveson1998}.
			A holomorphic reproducing kernel is unitarily invariant if
			\[
			K(Uz,Uw)=K(z,w)
			\]
			for every unitary map $U$ on $\mathbb C^d$.
			These spaces have been
			studied extensively; see, for example,
			\cite{GreeneRichterSundberg2002,GuoHuXu2004,Arveson2007,
	Hartz2017,BickelHartzMcCarthy2018,ClouatreHartz2018,
	AlemanHartzMcCarthyRichter2022,ChalmoukisHartz2022}
	and the references therein.
	In this paper, we give a direct construction in every dimension and
	establish a coefficient criterion for 
	unitarily invariant reproducing kernel Hilbert spaces on $\mathbb B_d$. 
	Let $K$ be a reproducing kernel on $\mathbb B_d$ of the form
	\[
	K(z,w)=\sum\limits_{m=0}^{\infty}a_m\langle z,w\rangle^m,
	\quad
	a_m>0, 
	\]
	and let $H=H(K)$ be the associated Hilbert space on
	$\mathbb B_d$. 
	Its homogeneous components are mutually
	orthogonal, and
	\[
	\|p\|_H^2=\frac{D_{N,d}}{a_N}\|p\|_{H^2(\partial\mathbb B_d)}^2,
	\qquad p\in\mathcal P_N^{(d)}.
	\]
	Our main result is the following theorem.
	
	\begin{theorem}\label{thm:mainK}
	Let $d\in\mathbb N$. Then the following are equivalent:
	\begin{enumerate}
		\item[\textup{(i)}]
		$\displaystyle\sup_{m\geq0}\frac{a_m}{D_{m,d}}<\infty$;
		\item[\textup{(ii)}]
		$H$ admits a uniformly bounded orthonormal basis
		consisting of homogeneous holomorphic polynomials.
	\end{enumerate}
	In this case, the basis $\{q_n\}_{n\geq1}$ can be chosen
	so that
	\[
	\sup_{n\geq1}\|q_n\|_\infty
	\leq
	\exp\left(\frac12d\log d+O(d)\right)
	\left(
	\sup_{m\geq0}\frac{a_m}{D_{m,d}}
	\right)^{\frac 12},
	\]
	where the implicit constant in $O(d)$ is independent of
	$d$ and the coefficient sequence $(a_m)$.
	\end{theorem}
	
	For the Hardy space, the reproducing kernel is
	\[
	K(z,w)
	=
	\frac{1}{(1-\langle z,w\rangle)^d}
	=
	\sum_{m=0}^{\infty}D_{m,d}\langle z,w\rangle^m.
	\]
	Thus Theorem~\ref{thm:mainK} gives an affirmative answer
	to Rudin's problem.
	
	\begin{cor}\label{cor:rudin}
	For every $d\in\mathbb N$ and $N\geq0$, the space
	$\mathcal P_N^{(d)}$ admits an orthonormal basis
	$\{p_{N,j}\}_{j=1}^{D_{N,d}}$ in $H^2(\partial \mathbb B_d)$
	satisfying
	\[
	\|p_{N,j}\|_\infty
	\leq \exp\left(\frac12d\log d+O(d)\right),
	\qquad 1\leq j\leq D_{N,d},
	\]
	where the implicit constant in $O(d)$ is independent of
	$d$, $N$, and $j$.
	\end{cor}
	Our construction works directly on the simplex of
	multi-indices.
	For $d\geq2$, we multiply the normalized monomials by
	a quadratic phase and combine them using the
	$(d-1)\times(d-1)$ minors of the normalized discrete
	Fourier matrix of order $N+d-1$.
	The Cauchy--Binet formula shows that these polynomials
	form an orthonormal basis.
	The same formula works for every degree $N$ and requires
	neither a piecewise rearrangement of the simplex nor
	an absorption argument.
	Expanding the determinants writes each basis polynomial
	as a finite linear combination of coordinate rotations
	of the single polynomial
	\[
	F_{N,d}(z)
	=
	\sum_{|\alpha|=N}
	\exp\left(2\pi i\sqrt2\sum_{k=1}^d\alpha_k^2\right)
	\sqrt{\frac{N!}{\alpha!}}\,z^\alpha,
	\]
	where the sum is over the whole simplex.
	Write $p_{N,\rho}$ for the basis polynomials constructed in Section~3.
	Since these rotations preserve the unit sphere,
	the determinant expansion gives
	\[
	\max_\rho\|p_{N,\rho}\|_\infty
	\leq
	\sqrt{(d-1)!}\,\|F_{N,d}\|_\infty.
	\]
	Thus a bound for this one polynomial controls the
	whole basis, without estimates for sums restricted
	to separate pieces of the simplex.
	
	To estimate $F_{N,d}$ on $\partial \mathbb B_d$, we write the square-root
	multinomial weights as a product of square-root Poisson
	weights with means $N|z_j|^2$, multiplied by a factor
	depending only on $N$.
	For $N\geq1$, this factor is
	$e^{N/2}\sqrt{N!}/N^{N/2}$.
	Keeping track of the dependence on $d$ in a weighted
	quadratic exponential sum estimate, we obtain
	\[
	\|F_{N,d}\|_\infty\leq A^d d^{\frac 14},
	\]
	where $A\geq1$ is independent of $d$ and $N$.
	The bound is independent of $N$ and holds uniformly
	on the whole sphere, including as some coordinates
	tend to zero.
	Combining the two estimates gives
	\[
	\max_{\rho}\|p_{N,\rho}\|_\infty
	\leq \exp\left(\frac12d\log d+O(d)\right),
	\]
	where the implicit constant in $O(d)$ is independent of $d$, $N$, and $\rho$.
	Taking these bases over all degrees gives an orthonormal
	basis of the Hardy space.
	This construction is the main step in the proof of
	Theorem~\ref{thm:mainK}.
	The homogeneous norm identity above then allows us
	to obtain an orthonormal basis of $H(K)$ by rescaling
	the Hardy basis in each degree.
	The coefficient condition in the theorem bounds these
	rescaling factors uniformly, so the resulting basis
	is uniformly bounded.
	
	This paper is organized as follows. In Section~\ref{sec:main},
	we  establish the lemmas
	needed for the proof. In Section~\ref{sec:proof}, we give the explicit construction and complete the proof of Theorem~\ref{thm:mainK}.

\section{Preliminaries}\label{sec:main}\label{sec:preliminaries}

Throughout the paper, unless otherwise specified, inner products on
$\mathcal P_N^{(d)}$ are inherited from $H^2(\partial\mathbb B_d)$.
For $t\in\R$, write $\|t\|_{\R/\Z}=\inf\limits_{k\in\Z}|t-k|$.

\begin{lemma}\label{lem:1d-intro}
	For fixed $q\in\mathbb Q\setminus\{0\}$, there exists a constant
	$C_q>0$, depending only on $q$, such that, for all $M\ge1$, $\eta\ge0$, and $c,b\in\R$,
	\[
	\sum\limits_{\substack{k\in\Z\\|k-c|\le \eta M}}
	\big(1+M\|q\sqrt2\,k+b\|_{\R/\Z}\big)^{-2}
	\le C_q(1+\eta).
	\]
\end{lemma}

\begin{proof}
	Write $q=\frac{a}{b_0}$, where $a\in\Z\setminus\{0\}$ and
	$b_0\in\N$. We first claim that there exists $c_q>0$ such that
	$\|q\sqrt2\,n\|_{\R/\Z}\ge \frac{c_q}{|n|}$ for every
	$n\in\Z\setminus\{0\}$.
	In fact, choose $\ell\in\Z$ such that
	$|q\sqrt2\,n-\ell|=\|q\sqrt2\,n\|_{\R/\Z}$.
	Since $2a^2n^2-b_0^2\ell^2$ is a nonzero integer, we have
	\begin{equation}
		1\le|2a^2n^2-b_0^2\ell^2|
		=|a\sqrt2\,n-b_0\ell|\,|a\sqrt2\,n+b_0\ell|
		=b_0\|q\sqrt2\,n\|_{\R/\Z}\,|a\sqrt2\,n+b_0\ell|.
		\label{eq:separation1}
	\end{equation}
	Moreover, since $|\ell|\le|q|\sqrt2\,|n|+\frac12$,
	\begin{equation}
		|a\sqrt2\,n+b_0\ell|
		\le2|a|\sqrt2\,|n|+\frac{b_0}{2}
		\le\left(2|a|\sqrt2+\frac{b_0}{2}\right)|n|.
		\label{eq:separation2}
	\end{equation}
	Therefore, by \eqref{eq:separation1} and \eqref{eq:separation2},
	\[
	\|q\sqrt2\,n\|_{\R/\Z}
	\ge\frac{1}{b_0\big(2|a|\sqrt2+\frac{b_0}{2}\big)|n|}.
	\]
	The claim holds by taking
	$c_q=\frac{1}{b_0\left(2|a|\sqrt2+\frac{b_0}{2}\right)}$.
	
	If $\eta=0$, the sum contains at most one term and is at most $1$.
	Suppose that $\eta>0$.
	For each integer $j\ge0$, set
	\[
	I_j^+=\left[\frac jM,\frac{j+1}{M}\right)+\Z,
	\qquad
	I_j^-=\left(-\frac{j+1}{M},-\frac jM\right]+\Z.
	\]
	Then
	\[
	\begin{aligned}
		A_j
		&=\left\{k\in\Z:|k-c|\le\eta M,\
		\frac jM\le\|q\sqrt2\,k+b\|_{\R/\Z}<\frac{j+1}{M}\right\}\\
		&\subset
		\bigcup_{\varepsilon\in\{+,-\}}
		\left\{k\in\Z:|k-c|\le\eta M,\
		q\sqrt2\,k+b+\Z\in I_j^\varepsilon\right\}.
	\end{aligned}
	\]
	Each arc $I_j^\pm$ has length at most $M^{-1}$.
	For $\varepsilon\in\{+,-\}$, set
	\[
	r_j^\varepsilon
	=\#\{k\in\Z:|k-c|\le\eta M,\
	q\sqrt2\,k+b+\Z\in I_j^\varepsilon\}.
	\]
	If $r_j^\varepsilon\ge2$, then, by the above claim, any two
	distinct integers counted by $r_j^\varepsilon$ satisfy
	\[
	\|q\sqrt2\,k-q\sqrt2\,k'\|_{\R/\Z}
	\ge\frac{c_q}{|k-k'|}
	\ge\frac{c_q}{2\eta M}.
	\]
	Thus
	\[
	\frac{(r_j^\varepsilon-1)c_q}{2\eta M}
	\le |I_j^\varepsilon|\le\frac1M.
	\]
	Thus $r_j^\varepsilon\le1+\frac{2\eta}{c_q}$, which also holds
	when $r_j^\varepsilon\le1$. 
	Therefore,
	\(
	\#A_j\le r_j^++r_j^-
	\le2\left(1+\frac{2\eta}{c_q}\right).
	\)
	Consequently,
	\[
	\begin{aligned}
		\sum\limits_{\substack{k\in\Z\\|k-c|\le \eta M}}
		\big(1+M\|q\sqrt2\,k+b\|_{\R/\Z}\big)^{-2}
		&=\sum\limits_{j=0}^\infty\sum\limits_{k\in A_j}
		\big(1+M\|q\sqrt2\,k+b\|_{\R/\Z}\big)^{-2}\\
		&\le\sum\limits_{j=0}^\infty\frac{\#A_j}{(1+j)^2}\\
		&\le2\left(1+\frac{2\eta}{c_q}\right)
		\sum\limits_{j=0}^\infty\frac{1}{(1+j)^2}\\
		&\le2\left(1+\frac{2}{c_q}\right)(1+\eta)
		\sum\limits_{j=0}^\infty\frac{1}{(1+j)^2}.
	\end{aligned}
	\]
	Taking $C_q=2\left(1+\frac{2}{c_q}\right)
	\sum\limits_{j=0}^\infty\frac{1}{(1+j)^2}$
	gives the desired conclusion.
\end{proof}

Let $r,s\in\N$ and write $\mathbf 1_r=(1,\dots,1)^T\in\R^r$. Set $A_{r,s}=I_r+s^{-1}\mathbf 1_r\mathbf 1_r^T$, where $I_r$ is the $r\times r$ identity matrix and \textit{The superscript $T$ denotes transpose.}

\begin{lemma}\label{lem:lattice-intro}
For every $R>0$ there exists $C_R>0$, depending only on $R$, such that, for all
$r,s\in\N$, $M=(M_1,\dots,M_r)$ with $M_i\ge1$, and
$b=(b_1,\dots,b_r)\in\R^r$,
\[
\sum\limits_{\substack{h\in\Z^r\\|h_i|\le RM_i}}
\prod\limits_{i=1}^r
\bigl(1+M_i\|2\sqrt2(A_{r,s}h)_i+b_i\|_{\R/\Z}\bigr)^{-2}
\le s C_R^r.
\]
\end{lemma}

\begin{proof}
	Since $A_{r,s}$ commutes with coordinate permutations, we assume that $M_1\ge M_2\ge\cdots\ge M_r$.
	Set $t_j=s+j-1$ and
	\begin{equation}\label{quant:Kdef}
		K_j=K_j(h)=(t_j+1)h_j+\sum\limits_{i=j+1}^r h_i,
		\ 1\le j\le r.
	\end{equation}
	By \eqref{quant:Kdef}, we have
	\begin{equation}\label{quant:Kcongruence}
		K_j-K_{j-1}=t_jh_j-(t_{j-1}+1)h_{j-1} = t_j(h_j-h_{j-1}),
		\ 2\le j\le r.
	\end{equation}
	Since $|h_i|\le RM_i$ and $M_i\le M_j$ for $i>j$,
	\begin{equation}\label{quant:Krange}
		|K_j|
		\le R\left((s+j)M_j+\sum\limits_{i=j+1}^rM_i\right)
		\le R(s+r)M_j.
	\end{equation}
	Therefore, by \eqref{quant:Kcongruence} and \eqref{quant:Krange},
	\begin{equation}\label{subset}
		\begin{aligned}
			&\left\{K(h)=(K_1(h),\ldots, K_r(h)):h\in\Z^r,\ |h_i|\le RM_i,\ 1\le i\le r\right\}\\
			\subset&
			\left\{K\in\Z^r:|K_j|\le R(s+r)M_j,\ 1\le j\le r,\
			K_j\equiv K_{j-1}\pmod{t_j},\ 2\le j\le r\right\}.
		\end{aligned}
	\end{equation}
	\noindent\textbf{Case 1:} $j\ge2$.
	Fix $K_1,\ldots,K_{j-1}$.
	By \eqref{quant:Kcongruence}, we write $K_j=K_{j-1}+t_ju$ with $u\in\Z$.
	By \eqref{quant:Kdef}, we have
	\begin{equation*}
		\begin{aligned}
			(A_{r,s}h)_j
			&=h_j+\frac1s\sum_{i=1}^r h_i\\
			&=h_j+\frac1{t_j}\sum_{i=j}^r h_i
			+\sum_{\ell=1}^{j-1}
			\left(\frac1{t_\ell}\sum_{i=\ell}^r h_i
			-\frac1{t_{\ell}+1}\sum_{i=\ell+1}^r h_i\right)\\
			&=\frac{K_j}{t_j}
			+\sum_{\ell=1}^{j-1}\frac{K_\ell}{t_\ell(t_\ell+1)}.
		\end{aligned}
	\end{equation*}
	Therefore,
	\begin{align}\label{app1}
	2\sqrt2(A_{r,s}h)_j+b_j=2\sqrt2\,u+\frac{2\sqrt2\,K_{j-1}}{t_j}
	+2\sqrt2\sum\limits_{\ell=1}^{j-1}
	\frac{K_\ell}{t_\ell(t_\ell+1)}+b_j.
	\end{align}
	Moreover, \eqref{quant:Krange} gives
	\begin{align}\label{app2}
	\left|u+\frac{K_{j-1}}{t_j}\right|
	\le\frac{R(s+r)}{t_j}M_j.
	\end{align}
	Write
	\[
	f_j(K_1,\ldots,K_j)
	=
	\left(
	1+M_j\left\|
	2\sqrt2\left(
	\frac{K_j}{t_j}
	+\sum_{\ell=1}^{j-1}\frac{K_\ell}{t_\ell(t_\ell+1)}
	\right)+b_j
	\right\|_{\mathbb R/\mathbb Z}
	\right)^{-2}.
	\]
	By \eqref{app1}, \eqref{app2} and taking $M=M_j$, $\eta=\frac{R(s+r)}{t_j}$,
	$c=-\frac{K_{j-1}}{t_j}$ and
	$b=\frac{2\sqrt2\,K_{j-1}}{t_j}
	+2\sqrt2\sum\limits_{\ell=1}^{j-1}
	\frac{K_\ell}{t_\ell(t_\ell+1)}+b_j$
	in Lemma~\ref{lem:1d-intro}, we obtain
	\begin{align}\label{sum1}
	\begin{aligned}
		\sum_{\substack{K_j\in\Z\\
				|K_j|\le R(s+r)M_j\\
				K_j\equiv K_{j-1}\pmod{t_j}}}
		f_j(K_1,\ldots,K_j)
		&=
		\sum_{\substack{u\in\Z\\
				\left|u+\frac{K_{j-1}}{t_j}\right|\le \frac{R(s+r)}{t_j}M_j}}
		\bigl(1+M_j\|2\sqrt2u +b\|_{\R/\Z}\bigr)^{-2}\\
		&\le C_{2}
		\left(1+\frac{R(s+r)}{t_j}\right).
	\end{aligned}
	\end{align}
	\medskip
	\noindent\textbf{Case 2:} $j=1$.
	Write $K_1=v+su$, where
	$v\in\{0,\ldots,s-1\}$ and $u\in\Z$.
	Since $t_1=s$, we have
	\[
	(A_{r,s}h)_1=\frac{K_1}{s}
	=u+\frac{v}{s},
	\qquad
	\left|u+\frac{v}{s}\right|
	\le\frac{R(s+r)}{s}M_1.
	\]
	Write
	\[
	f_1(K_1)
	=
	\left(
	1+M_1\left\|
	2\sqrt2\,\frac{K_1}{s}+b_1
	\right\|_{\mathbb R/\mathbb Z}
	\right)^{-2}.
	\]
	Therefore, for each $v\in\{0,\ldots,s-1\}$, taking $M=M_1$, $\eta=\frac{R(s+r)}{s}$,
	$c=-\frac{v}{s}$ and $b=\frac{2\sqrt2\,v}{s}+b_1$
	in Lemma~\ref{lem:1d-intro}, we obtain
	\begin{align}\label{sum2}
	\begin{aligned}
		\sum_{\substack{K_1\in\Z\\|K_1|\le R(s+r)M_1}}
		f_{1}(K_1)
		&=
		\sum_{v=0}^{s-1}
		\sum_{\substack{u\in\Z\\
				\left|u+\frac{v}{s}\right|\le\frac{R(s+r)}{s}M_1}}
		\bigl(1+M_1\|2\sqrt2\,u+b\|_{\R/\Z}\bigr)^{-2}\\
		&\le sC_{2}
		\left(1+\frac{R(s+r)}{s}\right).
	\end{aligned}
	\end{align}
	Therefore, by  \eqref{sum1} and \eqref{sum2},
	we obtain
	\[
	\begin{aligned}
		&\sum_{\substack{h\in\Z^r\\|h_i|\le RM_i}}
		\prod_{i=1}^r
		\bigl(1+M_i\|2\sqrt2(A_{r,s}h)_i+b_i\|_{\R/\Z}\bigr)^{-2}\\
		\overset{(*)}{\le}&
		\sum_{\substack{K_1\in\Z\\|K_1|\le R(s+r)M_1}}
		f_1(K_1)
		\sum_{\substack{K_2\in\Z\\|K_2|\le R(s+r)M_2\\
				K_2\equiv K_1\pmod{t_2}}}
		f_2(K_1,K_2)\cdots
		\sum_{\substack{K_r\in\Z\\|K_r|\le R(s+r)M_r\\
				K_r\equiv K_{r-1}\pmod{t_r}}}
		f_r(K_1,\ldots,K_r)\\
		\le&
		\left[\prod_{j=2}^r
		C_{2}\left(1+\frac{R(s+r)}{t_j}\right)\right]
		\sum_{\substack{K_1\in\Z\\|K_1|\le R(s+r)M_1}}
		f_1(K_1)\\
		\le&
		sC_{2}^r
		\prod_{j=1}^r
		\left(1+\frac{R(s+r)}{s+j-1}\right)\\
		\overset{(**)}{\leq}& sC_{2}^r [e(1+2R)]^r,
	\end{aligned}
	\]
	where \((*)\) follows from \eqref{subset} and the fact that $h\mapsto K(h)$ is injective, and
	\((**)\) follows from
	\[
	\prod_{j=1}^r
	\left(1+\frac{R(s+r)}{s+j-1}\right)
	\le\prod_{j=1}^r\left(1+\frac{R(r+1)}{j}\right)
	\le\prod_{j=1}^r\frac{j+2Rr}{j}
	\le(1+2R)^r\frac{r^r}{r!}
	\le[(1+2R)e]^r.
	\]
	Taking $C_R=(1+2R)eC_{2}$ gives the desired conclusion.

\end{proof}

For $x\in\R^r$ and $M\in(0,\infty)^r$, write $\frac{x}{M}=(\frac{x_1}{M_1},\dots,\frac{x_r}{M_r})$ and $Mx=(M_1x_1,\ldots,M_rx_r)$. For a multi-index $\nu\in\N_0^r$ write $\partial^\nu=\partial_1^{\nu_1}\cdots\partial_r^{\nu_r}$. We use the Fourier transform
\[
\widehat f(\xi)=\int_{\R^r}f(x)\e^{-2\pi ix\cdot\xi}\,dx.
\]

\begin{lemma}\label{lem:compact-intro}
Let $r\ge1$, $R>0$ and $a\in C_c^\infty(\R^r)$ with $\supp a\subset[-R,R]^r$. For $c\in\R^r$, $M\in(0,\infty)^r$ and $h\in\Z^r$, define
\[
b_h(x)=a\Big(\frac{x+h-c}{M}\Big)\overline{a\Big(\frac{x-c}{M}\Big)}.
\]
Then there exists a constant $D_R\ge1$, depending only on $R$, such that
\[
\big|\widehat{b_h}(\xi)\big|\le D_R^r\mathcal N_r(a)^2\left(\prod\limits_{i=1}^rM_i\right)\prod\limits_{i=1}^r(1+M_i|\xi_i|)^{-2},\ \xi\in\R^r,
\]
where $\mathcal N_r(a)=\max\limits_{\nu\in\{0,1,2\}^r}\|\partial^\nu a\|_\infty$.
\end{lemma}

\begin{proof}
For every $\nu\in\{0,1,2\}^r$, we have
\[
\partial^\nu b_h(x)=\left(\prod\limits_{i=1}^rM_i^{-\nu_i}\right)\sum\limits_{\alpha\le\nu}\binom{\nu}{\alpha}(\partial^\alpha a)\Big(\frac{x+h-c}{M}\Big)\overline{(\partial^{\nu-\alpha}a)\Big(\frac{x-c}{M}\Big)},
\]
where $\alpha\le\nu$ means $\alpha_i\le\nu_i$ for all $i$ and $\binom{\nu}{\alpha}=\prod\limits_{i=1}^r\binom{\nu_i}{\alpha_i}$. Since both $\alpha$ and $\nu-\alpha$ belong to $\{0,1,2\}^r$, we have
\begin{equation}
\begin{aligned}
\|\partial^\nu b_h\|_\infty
\le\left(\prod\limits_{i=1}^rM_i^{-\nu_i}\right)
\sum\limits_{\alpha\le\nu}\binom{\nu}{\alpha}
\|\partial^\alpha a\|_\infty\|\partial^{\nu-\alpha}a\|_\infty
\le2^{|\nu|}\mathcal N_r(a)^2\prod\limits_{i=1}^rM_i^{-\nu_i}
\le4^r\mathcal N_r(a)^2\prod\limits_{i=1}^rM_i^{-\nu_i}.
\end{aligned}\label{eq:derivativebound}
\end{equation}
Let $D_M=\prod\limits_{i=1}^r(1-M_i^2\partial_i^2)$.
Then, by \eqref{eq:derivativebound},
\begin{align}\label{2.20}
\|D_Mb_h\|_\infty
=
\left\|\sum\limits_{\epsilon\in\{0,1\}^r}(-1)^{|\epsilon|}\left(\prod\limits_{i=1}^rM_i^{2\epsilon_i}\right)\partial^{2\epsilon}b_h\right\|_\infty
\le
\sum\limits_{\epsilon\in\{0,1\}^r}\left(\prod\limits_{i=1}^rM_i^{2\epsilon_i}\right)\|\partial^{2\epsilon}b_h\|_\infty\le8^r\mathcal N_r(a)^2.
\end{align}
Since $\supp a\subset[-R,R]^r$, we have
\[
\supp(D_Mb_h)\subset\supp b_h\subset\prod\limits_{i=1}^r[c_i-RM_i,c_i+RM_i].
\]
It follows from \eqref{2.20} that
\begin{equation}
\|D_Mb_h\|_{L^1(\R^r)}\le(2R)^r\left(\prod\limits_{i=1}^rM_i\right)\|D_Mb_h\|_\infty\le 16^r R^r\mathcal N_r(a)^2\prod\limits_{i=1}^rM_i.\label{eq:L1bound}
\end{equation}
Writing $\mathrm e(t)=e^{2\pi it}$ for $t\in\mathbb R$ and applying integration by parts in each coordinate, we obtain
\begin{equation}
\widehat{D_Mb_h}(\xi)=\int_{\R^r}b_h(x)D_M(\e(-x\cdot\xi))\,dx=\left(\prod\limits_{i=1}^r(1+4\pi^2M_i^2\xi_i^2)\right)\widehat{b_h}(\xi).\label{eq:FourierDM}
\end{equation}
Therefore, by \eqref{eq:FourierDM}, \eqref{eq:L1bound} and the fact that $(1+y)^2\le2(1+y^2)\le2(1+4\pi^2y^2)$ for $y\ge0$, we have
\begin{equation*}
\begin{aligned}
|\widehat{b_h}(\xi)|&=|\widehat{D_Mb_h}(\xi)|\prod\limits_{i=1}^r(1+4\pi^2M_i^2\xi_i^2)^{-1}\\
&\le\|D_Mb_h\|_{L^1(\R^r)}\prod\limits_{i=1}^r(1+4\pi^2M_i^2\xi_i^2)^{-1}\\
&\le 16^r R^r\mathcal N_r(a)^2\left(\prod\limits_{i=1}^rM_i\right)\prod\limits_{i=1}^r(1+4\pi^2M_i^2\xi_i^2)^{-1}\\
&\le 32^r R^r\mathcal N_r(a)^2\left(\prod\limits_{i=1}^rM_i\right)\prod\limits_{i=1}^r(1+M_i|\xi_i|)^{-2}.
\end{aligned}
\end{equation*}
Thus $D_R=32(1+R)$ is admissible.
\end{proof}

\begin{lemma}\label{lem:expsum-intro}
Let $R$, $a$ and $\mathcal N_r(a)$ be as in Lemma \ref{lem:compact-intro}. Then there exists a constant $A_R\ge1$, depending only on $R$, such that, for every $s\in\mathbb N$, every $M=(M_1,\dots,M_r)$ with $M_i\ge1$ and every $c,\ell\in\R^r$,
\[
\Big|\sum\limits_{n\in\Z^r}a\Big(\frac{n-c}{M}\Big)\e\big(\sqrt2\,n^TA_{r,s}n+\ell\cdot n\big)\Big|\le\sqrt{s}A_R^r \mathcal N_r(a)\prod\limits_{i=1}^rM_i^{\frac 12},
\]
where $\ell\cdot n=\sum\limits_{i=1}^r\ell_in_i$.
\end{lemma}

\begin{proof}
Let
\[
\mathcal T=\sum\limits_{n\in\Z^r}a\Big(\frac{n-c}{M}\Big)\e\big(\sqrt2\,n^TA_{r,s}n+\ell\cdot n\big).
\]
Since $A_{r,s}$ is symmetric, for $n,h\in\Z^r$,
\begin{equation}
\begin{aligned}
(n+h)^TA_{r,s}(n+h)-n^TA_{r,s}n&=n^TA_{r,s}h+h^TA_{r,s}n+h^TA_{r,s}h\\
&=2h^TA_{r,s}n+h^TA_{r,s}h\\
&=2(A_{r,s}h)\cdot n+h^TA_{r,s}h.
\end{aligned}\label{eq:quadraticdiff}
\end{equation}
For $h\in\Z^r$, set $v_h=2\sqrt2\,A_{r,s}h$ and
\[
b_h(x)=a\Big(\frac{x+h-c}{M}\Big)\overline{a\Big(\frac{x-c}{M}\Big)}.
\]
If $b_h\not\equiv0$, there exists $x\in\R^r$ such that
$a(\frac{x+h-c}{M})\ne0$ and $a(\frac{x-c}{M})\ne0$.
Since $\supp a\subset[-R,R]^r$, for any $1\le i\le r$, we have
\begin{equation}
|h_i|\le|x_i+h_i-c_i|+|x_i-c_i|\le2RM_i.\label{eq:hrange}
\end{equation}
Then, writing $m=n+h$, by \eqref{eq:quadraticdiff}, \eqref{eq:hrange} and Lemma \ref{lem:compact-intro}, we have
\begin{equation}
\begin{aligned}
|\mathcal T|^2&=\sum\limits_{m,n\in\Z^r}a\Big(\frac{m-c}{M}\Big)\overline{a\Big(\frac{n-c}{M}\Big)}\e\big(\sqrt2(m^TA_{r,s}m-n^TA_{r,s}n)+\ell\cdot(m-n)\big)\\
&=\sum\limits_{h\in\Z^r}\sum\limits_{n\in\Z^r}b_h(n)\e\big(\sqrt2(2(A_{r,s}h)\cdot n+h^TA_{r,s}h)+\ell\cdot h\big)\\
&=\sum\limits_{h\in\Z^r}\e\big(\sqrt2\,h^TA_{r,s}h+\ell\cdot h\big)\sum\limits_{n\in\Z^r}b_h(n)\e(v_h\cdot n)\\
&\le\sum\limits_{h\in\Z^r}\Big|\sum\limits_{n\in\Z^r}b_h(n)\e(v_h\cdot n)\Big|\\
&=\sum\limits_{\substack{h\in\Z^r\\|h_i|\le2RM_i\ (1\le i\le r)}}\Big|\sum\limits_{n\in\Z^r}b_h(n)\e(v_h\cdot n)\Big|\\
&\overset{(*)}{=}\sum\limits_{\substack{h\in\Z^r\\|h_i|\le2RM_i\ (1\le i\le r)}}\Big|\sum\limits_{k\in\Z^r}\widehat{b_h}(k-v_h)\Big|\\
&\le D_R^r \mathcal N_r(a)^2\left(\prod\limits_{i=1}^rM_i\right)\sum\limits_{\substack{h\in\Z^r\\|h_i|\le2RM_i\ (1\le i\le r)}}\prod\limits_{i=1}^r\sum\limits_{k_i\in\Z}(1+M_i|k_i-(v_h)_i|)^{-2},
\end{aligned}\label{eq:T2}
\end{equation}
where \((*)\) follows from the Poisson summation formula. Next, we claim that there exists a constant $C>0$ such that, for all $M\ge1$ and $t\in\R$,
\begin{equation}
\sum\limits_{k\in\Z}(1+M|k-t|)^{-2}\le C(1+M\|t\|_{\R/\Z})^{-2}.\label{eq:periodization}
\end{equation}
In fact, choose $k_0\in\Z$ such that $|t-k_0|=\|t\|_{\R/\Z}\le \frac12$. For every $j\in\Z\setminus\{0\}$, we have
\begin{equation}
|k_0+j-t|\ge|j|-|t-k_0|\ge|j|-\frac12\ge\frac{|j|}{2}.\label{eq:revtriangle}
\end{equation}
Since \(M \geq 1\) and \(\|t\|_{\mathbb R/\mathbb Z}\leq \frac12\),
\begin{equation}
\frac{1}{M^2}
\le\frac94\left(1+\frac M2\right)^{-2}
\le\frac94\bigl(1+M\|t\|_{\R/\Z}\bigr)^{-2},
\qquad M\ge1.\label{eq:M2bound}
\end{equation}
Therefore, by \eqref{eq:revtriangle} and \eqref{eq:M2bound},
\[
\begin{aligned}
\sum\limits_{k\in\Z}(1+M|k-t|)^{-2}&=(1+M|k_0-t|)^{-2}+\sum\limits_{j\in\Z\setminus\{0\}}(1+M|k_0+j-t|)^{-2}\\
&\le(1+M\|t\|_{\R/\Z})^{-2}+\sum\limits_{j\in\Z\setminus\{0\}}\Big(1+\frac{M|j|}{2}\Big)^{-2}\\
&\le(1+M\|t\|_{\R/\Z})^{-2}+\frac{4}{M^2}\sum\limits_{j\in\Z\setminus\{0\}}\frac{1}{j^2}\\
&\le\left(1+9\sum\limits_{j\in\Z\setminus\{0\}}\frac{1}{j^2}\right)(1+M\|t\|_{\R/\Z})^{-2}.
\end{aligned}
\]
The claim is proved. Hence, by \eqref{eq:T2}, \eqref{eq:periodization} and Lemma \ref{lem:lattice-intro} applied with $R$ replaced by $2R$ and $b=0$,
\[
\begin{aligned}
|\mathcal T|^2&\le\left(C D_R\right)^r \mathcal N_r(a)^2\left(\prod\limits_{i=1}^rM_i\right)\sum\limits_{\substack{h\in\Z^r\\|h_i|\le2RM_i}}\prod\limits_{i=1}^r\big(1+M_i\|2\sqrt2(A_{r,s}h)_i\|_{\R/\Z}\big)^{-2}\\
&\le s \left(C D_RC_{2R}\right)^r \mathcal N_r(a)^2\left( \prod\limits_{i=1}^rM_i\right).
\end{aligned}
\]
Taking \(A_R = \max\left\{1,\sqrt{CD_RC_{2R}}\right\}\) proves the asserted estimate.
\end{proof}

Let $\mathcal S(\mathbb R^r)$ denote the Schwartz space
of all functions $a\in C^\infty(\mathbb R^r)$ such that, for every \(K\in\mathbb N_0\) and \(\nu\in\mathbb N_0^r\),
\[
\sup_{u\in\mathbb R^r}
(1+|u|)^K|\partial^\nu a(u)|<\infty,
\]
where $|u|=\left(\sum\limits_{i=1}^r u_i^2\right)^{\frac 12}$.

\begin{proposition}\label{prop:tensor-quant}
	Let $r\in\N$ and $g_1,\ldots,g_r\in\mathcal S(\R)$ satisfy
	\[
	\max_{1\le i\le r}\max_{0\le j\le2}
	\sup_{u\in\R}(1+|u|)^2|g_i^{(j)}(u)|
	\le B,\ B\ge1.
	\]
	Then there exists $C\ge1$, independent of all the above parameters, such that, for all $M_i\ge1$ and $c,\ell\in\R^r$,
	\begin{equation*}
		\left|\sum_{n\in\Z^r}\prod_{i=1}^r
		g_i\left(\frac{n_i-c_i}{M_i}\right)
		\e\bigl(\sqrt2\,n^TA_{r,1}n+\ell\cdot n\bigr)\right|
		\le(CB)^r\prod_{i=1}^rM_i^{\frac 12}.
	\end{equation*}
\end{proposition}

\begin{proof}
	Set $a(u)=\prod\limits_{i=1}^rg_i(u_i)$.
	Choose a fixed nonnegative function $\psi\in C_c^\infty(\R)$
	such that $\supp\psi\subset[-1,1]$ and
	\begin{equation}\label{sum=1}
		\sum_{j\in\Z}\psi(t-j)=1,\ t\in\R.
	\end{equation}
	For $k\in\Z^r$, define
	\[
	a_k(u)=a(u+k)\prod_{i=1}^r\psi(u_i),
	\ u\in\R^r.
	\]
	Then $a_k\in C_c^\infty(\R^r)$ and
	$\supp a_k\subset[-1,1]^r$.
	Moreover, by \eqref{sum=1},
	\begin{equation}\label{eq:schwartz-decomposition}
		\begin{aligned}
			\sum_{k\in\Z^r}a_k(u-k)
			=a(u)\sum_{k\in\Z^r}\prod_{i=1}^r\psi(u_i-k_i)
			=a(u)\prod_{i=1}^r
			\left(\sum_{k_i\in\Z}\psi(u_i-k_i)\right)
			=a(u).
		\end{aligned}
	\end{equation}
	Let
	\[
	T'=\sum_{n\in\Z^r}
	a\left(\frac{n-c}{M}\right)
	\e\left(\sqrt2\,n^TA_{r,1}n+\ell\cdot n\right).
	\]
	By \eqref{eq:schwartz-decomposition} and
	Lemma~\ref{lem:expsum-intro} applied with $R=s=1$
	and $c$ replaced by $c+Mk$, we have
	\begin{equation}\label{eq:tensor-localization}
		\begin{aligned}
			|T'|
			&=\left|\sum_{k\in\Z^r}\sum_{n\in\Z^r}
			a_k\left(\frac{n-(c+Mk)}{M}\right)
			\e\left(\sqrt2\,n^TA_{r,1}n+\ell\cdot n\right)\right|\\
			&\leq\sum_{k\in\Z^r}\left|\sum_{n\in\Z^r}
			a_k\left(\frac{n-(c+Mk)}{M}\right)
			\e\left(\sqrt2\,n^TA_{r,1}n+\ell\cdot n\right)\right|\\
			&\le A_1^r\left(\prod_{i=1}^rM_i^{\frac 12}\right)
			\sum_{k\in\Z^r}\mathcal N_r(a_k),
		\end{aligned}
	\end{equation}
	where $A_1\ge1$ is the constant in Lemma~\ref{lem:expsum-intro} with $R=1$.
	We claim that there exists a constant $C_0\ge1$, independent of $r$, $B$ and $k$,
	such that
	\[
	\mathcal N_r(a_k)\le(C_0B)^r
	\prod_{i=1}^r(1+|k_i|)^{-2},
	\ k\in\Z^r.
	\]
	In fact, for $u\in[-1,1]^r$, we have
	\[
	1+|k_i|
	\le1+|u_i+k_i|+|u_i|
	\le2(1+|u_i+k_i|),
	\ 1\le i\le r.
	\]
	Consequently,
	\[
	(1+|u_i+k_i|)^{-2}
	\le4(1+|k_i|)^{-2}.
	\]
	Therefore, for $\mu\in\{0,1,2\}^r$, by the assumptions on $g_i$,
	\begin{equation}\label{eq:tensor-partial}
		\begin{aligned}
			|(\partial^\mu a)(u+k)|
			=\prod_{i=1}^r|g_i^{(\mu_i)}(u_i+k_i)|
			\le B^r\prod_{i=1}^r(1+|u_i+k_i|)^{-2}
			\le(4B)^r\prod_{i=1}^r(1+|k_i|)^{-2}.
		\end{aligned}
	\end{equation}
	For $\nu\in\{0,1,2\}^r$, we have
	\begin{equation}\label{eq:tensor-leibniz}
		\partial^\nu a_k(u)
		=\sum_{\mu\le\nu}\binom{\nu}{\mu}
		(\partial^\mu a)(u+k)
		\prod_{i=1}^r\psi^{(\nu_i-\mu_i)}(u_i).
	\end{equation}
	Let $C_\psi=\max\left\{1,\max\limits_{0\le q\le2}
	\|\psi^{(q)}\|_\infty\right\}$.
	Then, by \eqref{eq:tensor-leibniz} and \eqref{eq:tensor-partial},
	we obtain
	\[
	\begin{aligned}
		\|\partial^\nu a_k\|_\infty
		\le C_\psi^r
		\sum_{\mu\le\nu}\binom{\nu}{\mu}
		\sup_{u\in[-1,1]^r}|(\partial^\mu a)(u+k)|
		\le(16C_\psi B)^r
		\prod_{i=1}^r(1+|k_i|)^{-2}.
	\end{aligned}
	\]
	The claim holds by taking $C_0=16C_\psi$.
	Hence, by \eqref{eq:tensor-localization} and the above claim,
	\[
	\begin{aligned}
		|T'|
		&\le(A_1C_0B)^r
		\left(\prod_{i=1}^rM_i^{\frac 12}\right)
		\sum_{k\in\Z^r}\prod_{i=1}^r(1+|k_i|)^{-2}\\
		&=(A_1C_0B)^r
		\left(\prod_{i=1}^rM_i^{\frac 12}\right)
		\left(\sum_{m\in\Z}(1+|m|)^{-2}\right)^r.
	\end{aligned}
	\]
	Taking $C=A_1C_0\sum\limits_{m\in\Z}(1+|m|)^{-2}$
	gives the desired conclusion.
\end{proof}
Let \(\Gamma\) denote the Gamma function. For $\lambda\geq1$ and $u\in\mathbb R$ such that
$x=\lambda+\sqrt{\lambda}\,u>- \frac12$, define
\[
G_\lambda(u)=\lambda^{\frac 14}\e^{-\frac{\lambda}{2}}
\frac{\lambda^{\frac{x}{2}}}{\sqrt{\Gamma(x+1)}}.
\]
By Stirling's formula \cite{Olver1997}, there exist constants
$c_\Gamma,C_\Gamma>0$ such that
\[
c_\Gamma t^{t-\frac12}e^{-t}
\leq\Gamma(t)
\leq C_\Gamma t^{t-\frac12}e^{-t},
\ t\geq\frac12.
\]
Therefore,
\begin{equation}\label{eq:Gestimate}
	\begin{aligned}
		G_\lambda(u)
		&\leq C\lambda^{\frac x2+\frac14}
		(x+1)^{-\frac x2-\frac14}e^{\frac{x+1-\lambda}{2}}\\
		&=C\left(\frac{x+1}{\lambda}\right)^{\frac14}
		\exp\left(
		-\frac12\left[
		(x+1)\log\frac{x+1}{\lambda}-(x+1)+\lambda
		\right]\right),
	\end{aligned}
\end{equation}
where $C>0$ is independent of $\lambda$ and $u$.
Write
\[
L_\lambda(u)
=\log G_\lambda(u)
=\frac14\log\lambda-\frac{\lambda}{2}
+\frac{x}{2}\log\lambda-\frac12\log\Gamma(x+1).
\]
Let $\psi=\Gamma'/\Gamma$.
Then
\begin{equation}\label{eq:Lderiv}
	L_\lambda'(u)
	=\frac{\sqrt\lambda}{2}
	\bigl(\log\lambda-\psi(x+1)\bigr),
	\qquad
	L_\lambda^{(q)}(u)
	=-\frac{\lambda^{\frac q2}}{2}\psi^{(q-1)}(x+1),
	\ q\geq2.
\end{equation}
Define polynomials $P_j$ recursively by $P_0=1$ and
\begin{equation}
P_{j+1}(X_1,\dots,X_{j+1})=X_1P_j(X_1,\dots,X_j)+\sum\limits_{q=1}^jX_{q+1}\frac{\partial P_j}{\partial X_q}(X_1,\dots,X_j),\qquad j\ge0,\label{eq:Pjrecursion}
\end{equation}
where the sum is understood to be zero when $j=0$.

\begin{lemma}\label{lem:Gderivative}
For $j\ge0$,
\[
G_\lambda^{(j)}=G_\lambda P_j(L_\lambda',\dots,L_\lambda^{(j)}).
\]
\end{lemma}

\begin{proof}
We prove the assertion by induction on $j$. The case $j=0$ follows from $P_0=1$. Suppose that the assertion holds for some $j\ge0$. Since $G_\lambda'=G_\lambda L_\lambda'$, we have
\[
G_\lambda^{(j+1)}=G_\lambda L_\lambda'P_j(L_\lambda',\dots,L_\lambda^{(j)})+G_\lambda\sum\limits_{q=1}^j\frac{\partial P_j}{\partial X_q}(L_\lambda',\dots,L_\lambda^{(j)})L_\lambda^{(q+1)}=G_\lambda P_{j+1}(L_\lambda',\dots,L_\lambda^{(j+1)}).
\]
\end{proof}

\begin{lemma}\label{lem:Pjmonomials}
For every $j\ge1$,  every monomial $X_1^{\alpha_1}\cdots X_j^{\alpha_j}$ occurring in $P_j$ satisfies $\sum\limits_{r=1}^jr\alpha_r=j$.
\end{lemma}

\begin{proof}
We prove the assertion by induction on $j$. The assertion holds for $P_1(X_1)=X_1$. Suppose that it holds for $P_j$ and set $\alpha_{j+1}=0$. For every monomial $X^\alpha$ occurring in $P_j$,
\[
X_1X^\alpha=X^{\alpha+\epsilon_1},\qquad X_{q+1}\frac{\partial X^\alpha}{\partial X_q}=\alpha_qX^{\alpha-\epsilon_q+\epsilon_{q+1}},\ \alpha_q\ge1,
\]
where $\epsilon_q=(0,\dots,0,\underset{q\text{th}}{1},0,\dots,0)\in\R^{j+1}$. Moreover,
\[
\sum\limits_{r=1}^{j+1}r(\alpha+\epsilon_1)_r=\sum\limits_{r=1}^{j+1}r\alpha_r+1=\sum\limits_{r=1}^jr\alpha_r+1=j+1,
\]
and
\[
\sum\limits_{r=1}^{j+1}r(\alpha-\epsilon_q+\epsilon_{q+1})_r=\sum\limits_{r=1}^{j+1}r\alpha_r-q+(q+1)=\sum\limits_{r=1}^jr\alpha_r+1=j+1.
\]
Thus, by \eqref{eq:Pjrecursion}, every monomial $X^\beta$ occurring in $P_{j+1}$ satisfies $\sum\limits_{r=1}^{j+1}r\beta_r=j+1$.
\end{proof}

\begin{lemma}\label{lem:Gbounds}
For every $K,j\in\N_0$,
\[
\sup\limits_{\substack{\lambda\ge1\\\lambda+\sqrt\lambda\,u>-\frac12}}(1+|u|)^K|G_\lambda^{(j)}(u)|<\infty.
\]
\end{lemma}

\begin{proof}
	Let $t=\lambda+\sqrt\lambda\,u+1>\frac12$ and
	$y=\frac{t-\lambda}{\sqrt\lambda}$. 
	Since $\lambda\ge1$, we have
	\begin{align}\label{222}
		t\log\frac t\lambda-t+\lambda
		=(t-\lambda)^2
		\int_0^1\frac{1-s}{\lambda+s(t-\lambda)}\,ds
		\ge\frac{(t-\lambda)^2}{2(\lambda+|t-\lambda|)}
		\ge\frac{y^2}{2(1+|y|)}
		\ge\frac{|y|-1}{2}.
	\end{align}
	Moreover, $|y-u|=\lambda^{-\frac12}\le1$ and
	$\frac t \lambda\le1+|y|\le2+|u|$. 
	By \eqref{eq:Gestimate} and \eqref{222},
	\begin{equation}\label{eq:Gexponential}
	\begin{aligned}
		G_\lambda(u)
		\le C\left(\frac t\lambda\right)^{\frac14}
		\exp\left(-\frac12\left(t\log\frac t\lambda-t+\lambda\right)\right)
		\le C(2+|u|)^{\frac14}\e^{-\frac{|u|}{4}}
		\le C\e^{-\frac{|u|}{8}}.
	\end{aligned}
\end{equation}
We claim that, for \(q\ge1\),
\begin{equation}\label{eq:Lpolynomial}
	|L_\lambda^{(q)}(u)|
	\le C_q(1+|u|)^{2q}.
\end{equation}
	To prove this, we distinguish two cases.
	\medskip
	\noindent\textbf{Case 1:} $t\ge \frac \lambda 2$.
	In this case, we have
	\begin{equation}
		\Big|\log\frac{\lambda}{t}\Big|=\int_{\min\{\lambda,t\}}^{\max\limits\{\lambda,t\}}\frac{d\zeta}{\zeta}\le\frac{2|t-\lambda|}{\lambda}.\label{eq:logbound}
	\end{equation}
	By \cite[Eq.~(1.2)]{GuoQiZhaoLuo2015}, we have
	\begin{equation}\label{eq:digammabound}
		0\leq\log \eta-\psi(\eta)\leq\frac1\eta,
		\ \eta>0,
	\end{equation}
	Hence, by  \eqref{eq:Lderiv}, \eqref{eq:digammabound} and \eqref{eq:logbound},
	\begin{equation*}
		|L_\lambda'(u)|\le\frac{\sqrt\lambda}{2}\Big(\Big|\log\frac{\lambda}{t}\Big|+\frac{1}{t}\Big)\le\frac{|t-\lambda|}{\sqrt\lambda}+\frac{1}{\sqrt\lambda}\le|u|+2\le 2(1+|u|)^2.\label{eq:Lprimebound}
	\end{equation*}
	For every integer $q\geq2$, by \cite[Eq.~(1.3)]{GuoQiZhaoLuo2015}, we have
	\begin{equation}\label{eq:psibound}
		|\psi^{(q-1)}(\eta)|\leq C_q \eta^{1-q},
		\ \eta\geq \frac 12,
	\end{equation}
	where \(C_q\) depends only on \(q\).
	Then, by \eqref{eq:Lderiv} and  \eqref{eq:psibound},
	\begin{equation*}
		|L_\lambda^{(q)}(u)|\le C_q\lambda^{\frac q2}t^{1-q}\le C_q\lambda^{1-\frac q2}\le C_q \le C_q(1+|u|)^{2q}.
	\end{equation*}
	\noindent\textbf{Case 2:} $t<\frac{\lambda}{2}$.
	In this case,
	\(|u|=\frac{\lambda-t+1}{\sqrt\lambda}
	\ge\frac{\sqrt\lambda}{2}\).
	Since \(t >\frac12\), by \eqref{eq:Lderiv} and \eqref{eq:digammabound},
	\[
	\begin{aligned}
		|L_\lambda'(u)|
		\le\frac{\sqrt\lambda}{2}
		\Big(\log\frac{\lambda}{t}+\frac1t\Big)
		\le C\sqrt\lambda(1+\log\lambda)
		\le C\lambda
		\le C(1+|u|)^2.
	\end{aligned}
	\]
	Then, by \eqref{eq:psibound} and \eqref{eq:Lderiv},
	for every integer $q\ge2$,
	\[
	|L_\lambda^{(q)}(u)|
	\le C_q\lambda^{\frac q2}t^{1-q}
	\le C_q\lambda^{\frac q2}
	\le C_q(1+|u|)^q
	\le C_q(1+|u|)^{2q}.
	\]
	The claim is proved.
	
	By Lemma \ref{lem:Pjmonomials}, every monomial
	$X_1^{\alpha_1}\cdots X_j^{\alpha_j}$ occurring in $P_j$
	satisfies $\sum\limits_{q=1}^j q\alpha_q=j$. Thus, by
	\eqref{eq:Lpolynomial},
	\[
	\prod_{q=1}^j|L_\lambda^{(q)}(u)|^{\alpha_q}
	\le\left(\prod_{q=1}^j C_q^{\alpha_q}\right)
	(1+|u|)^{2\sum\limits_{q=1}^j q\alpha_q}
	\le C_j(1+|u|)^{2j}.
	\]
	Since $P_j$ has finitely many terms and its coefficients
	depend only on $j$, it follows that
	\begin{align}\label{pp}
	|P_j(L_\lambda'(u),\dots,L_\lambda^{(j)}(u))|
	\le C_j(1+|u|)^{2j}.
	\end{align}
	Therefore, by Lemma \ref{lem:Gderivative}, \eqref{pp} and
	\eqref{eq:Gexponential}, we obtain
	\[
	|G_\lambda^{(j)}(u)|
	\le C_j(1+|u|)^{2j}\e^{-\frac{|u|}{8}},
	\ j\in\N_0,
	\]
	where the case $j=0$ follows directly from
	\eqref{eq:Gexponential}. Therefore, for every $K,j\in\N_0$,
	\[
	(1+|u|)^K|G_\lambda^{(j)}(u)|
	\le C_j(1+|u|)^{K+2j}\e^{-\frac{|u|}{8}}
	\le C_{K,j}.
	\]
\end{proof}

For $\lambda\ge0$, define $f_\lambda:\N_0\to[0,\infty)$ by
\begin{equation*}
f_\lambda(k)=\begin{cases}
\e^{-\frac \lambda 2}\dfrac{\lambda^{\frac k2}}{\sqrt{k!}},&\lambda>0,\ k\ge0,\\
1,&\lambda=0,\ k=0,\\
0,&\lambda=0,\ k\ge1.
\end{cases}
\end{equation*}
Next, for $\lambda\ge1$, we extend $f_\lambda$ to a smooth function on $\R$ that vanishes at every negative integer. Let
\[
\vartheta(v)=\begin{cases}\e^{-\frac 1v},&v>0,\\0,&v\le0,\end{cases}\qquad
\chi(x)=\frac{\vartheta(x+\frac 12)}{\vartheta(x+ \frac 12)+\vartheta(-x)}.
\]
It is easy to verify that $\chi\in C^\infty(\R)$. Moreover, $0\le\chi\le1$ and
\begin{equation}
\chi(x)=\begin{cases}0,&x\le- \frac 12,\\1,&x\ge0.\end{cases}\label{eq:chi}
\end{equation}
For $\lambda\geq1$, use the same notation for the extension
\[
f_\lambda(x)=
\begin{cases}
\displaystyle\chi(x)\e^{- \frac \lambda2}
\frac{\lambda^{\frac x2}}{\sqrt{\Gamma(x+1)}},&x>-\frac12,\\
0,&x\leq-\frac12.
\end{cases}
\]

\begin{lemma}\label{lem:fsmooth}
For every $\lambda\ge1$, $f_\lambda\in C^\infty(\R)$. Moreover,
\[
f_\lambda(k)=
\begin{cases}
\e^{-\frac \lambda 2}\dfrac{\lambda^{\frac k2}}{\sqrt{k!}},&k\in\N_0,\\
0,&k\in\Z,\ k<0.
\end{cases}
\]
\end{lemma}

\begin{proof}
Since $\Gamma$ is positive and smooth on $(0,\infty)$,
\[
h_\lambda(x)=\e^{-\frac \lambda 2}\frac{\lambda^{\frac x2}}{\sqrt{\Gamma(x+1)}}\in C^\infty((-1,\infty)).
\]
Thus $f_\lambda\in C^\infty(\R\setminus\{-\frac12\})$. For every $q\in\N_0$,  it is easy to verify that \( \chi^{(q)}\Big(-\frac12\Big)=0\).
Therefore,
\[
\lim_{x\to(-\frac12)^+}\frac{d^q}{dx^q}\big(\chi(x)h_\lambda(x)\big)=\sum\limits_{j=0}^q\binom{q}{j}\chi^{(j)}\left(-\frac 12\right)h_\lambda^{(q-j)}\left(-\frac 12\right)=0.
\]
Thus $f_\lambda\in C^\infty(\R)$. For every $k\in\N_0$, we have $\chi(k)=1$ and $\Gamma(k+1)=k!$, so $f_\lambda(k)=\e^{-\frac \lambda 2}\dfrac{\lambda^{\frac k2}}{\sqrt{k!}}$. For every negative integer $k$, we have $k\le-\frac 12$, and hence $f_\lambda(k)=0$.
\end{proof}

\begin{proposition}\label{lem:profiles-intro}
For $\lambda\ge1$, define
\[
g_\lambda(u)=\lambda^{\frac14}f_\lambda(\lambda+\sqrt\lambda\,u),\ u\in\R.
\]
Then, for every $K,j\in\N_0$, there exists a constant $C_{K,j}>0$, independent of $\lambda$, such that
\begin{equation}\label{star}
\sup\limits_{\lambda\ge1}\sup\limits_{u\in\R}(1+|u|)^K|g_\lambda^{(j)}(u)|\le C_{K,j}.
\end{equation}
In particular, $\{g_\lambda:\lambda\ge1\}$ is a bounded subset of $\mathcal S(\R)$.
\end{proposition}

\begin{proof}
	Put $x=\lambda+\sqrt{\lambda}u$.
	For $x\leq-\frac12$ and $j\in\N_0$,
	\[
	g_\lambda^{(j)}(u)
	=\lambda^{\frac14+\frac j2}f_\lambda^{(j)}(x)=0.
	\]
	Hence \eqref{star} holds.
	It remains to consider the case $x>-\frac12$.
	In this case, $g_\lambda(u)=\chi(x)G_\lambda(u)$.
	For $1\leq r\leq j$, by \eqref{eq:chi}, $\supp\chi^{(r)}\subset[- \frac 12,0]$.
	Moreover, for $-\frac 12<x\leq0$, we have
	\begin{equation}\label{eq:localbound}
		\sqrt{\lambda}
		\leq\frac{\lambda-x}{\sqrt{\lambda}}
		=|u|\leq1+|u|.
	\end{equation}
	Therefore, for $x>-\frac 12$, by \eqref{eq:localbound} and
	Lemma~\ref{lem:Gbounds},
	\[
	\begin{aligned}
		(1+|u|)^K|g_\lambda^{(j)}(u)|
		&\leq(1+|u|)^K|\chi(x)G_\lambda^{(j)}(u)|
		+\sum_{r=1}^j\binom{j}{r}
		(1+|u|)^K\lambda^{\frac r2}
		|\chi^{(r)}(x)G_\lambda^{(j-r)}(u)|\\
		&\leq(1+|u|)^K|G_\lambda^{(j)}(u)|
		+C_j\sum_{r=1}^j
		(1+|u|)^{K+r}|G_\lambda^{(j-r)}(u)|\\
		&\leq C_{K,j}.
	\end{aligned}
	\]
	
\end{proof}

Let
\[
\Delta_{d-1}=\Big\{p=(p_1,\dots,p_d)\in[0,1]^d:\sum\limits_{j=1}^dp_j=1\Big\}.
\]
For $N\in\mathbb N_0$, set
\[
\Lambda_{N,d}=\{\alpha\in\mathbb N_0^d:|\alpha|=N\}.
\]
For $p\in\Delta_{d-1}$ and $\alpha\in\Lambda_{N,d}$,
define $w_{N,p}(\alpha)=\sqrt{\frac{N!}{\alpha!}}\,p^{\frac \alpha 2}$.
Then
\begin{equation}
\sum\limits_{\alpha\in\Lambda_{N,d}}w_{N,p}(\alpha)^2=\sum\limits_{\substack{\alpha\in\N_0^d\\|\alpha|=N}}\frac{N!}{\alpha!}\prod\limits_{j=1}^dp_j^{\alpha_j}=\left(\sum\limits_{j=1}^dp_j\right)^N=1.\label{eq:multinomial}
\end{equation}
For $\xi=(\xi_1,\dots,\xi_d)\in\R^d$, set
\begin{equation}
S_N(p,\xi)=\sum\limits_{\alpha\in\Lambda_{N,d}}w_{N,p}(\alpha)\e\left(\sqrt2\sum\limits_{j=1}^d\alpha_j^2+\xi\cdot\alpha\right),\label{eq:SN}
\end{equation}
where $\xi\cdot\alpha=\sum\limits_{j=1}^d\xi_j\alpha_j$.

\begin{lemma}\label{lem:Fsup}
For every $N\in\N$, we have
\[
\|F_{N,d}\|_\infty=\sup\limits_{\substack{p\in\Delta_{d-1}\\\xi\in\R^d}}|S_N(p,\xi)|.
\]
\end{lemma}

\begin{proof}
For $z\in\partial\mathbb B_d$, set $p_j=|z_j|^2$. Since $\sum\limits_{j=1}^d|z_j|^2=1$, we have $p=(p_1,\dots,p_d)\in\Delta_{d-1}$. If $z_j\ne0$, choose $\xi_j\in\R$ such that $z_j=|z_j|\e(\xi_j)$. If $z_j=0$, take $\xi_j=0$. Then $z_j=\sqrt{p_j}\e(\xi_j)$ for every $j$, and hence
\begin{align}\label{111}
z^\alpha=\prod\limits_{j=1}^d\bigl(\sqrt{p_j}\e(\xi_j)\bigr)^{\alpha_j}=p^{\frac\alpha2}\e(\xi\cdot\alpha).
\end{align}
By the definitions of $F_{N,d}$ and \eqref{111}, we have \(F_{N,d}(z)=S_N(p,\xi)\),
which implies that 
\[
\|F_{N,d}\|_\infty\leq\sup\limits_{\substack{p\in\Delta_{d-1}\\\xi\in\R^d}}|S_N(p,\xi)|.
\]
Conversely, for any $p\in\Delta_{d-1}$ and $\xi\in\R^d$, set $z_j=\sqrt{p_j}\e(\xi_j)$. Then $\sum\limits_{j=1}^d|z_j|^2=\sum\limits_{j=1}^dp_j=1$, so $z\in\partial\mathbb B_d$. 
By the definitions of $F_{N,d}$ and \eqref{111}, we obtain $S_N(p,\xi)=F_{N,d}(z)$, which implies that 
\[
\|F_{N,d}\|_\infty\geq\sup\limits_{\substack{p\in\Delta_{d-1}\\\xi\in\R^d}}|S_N(p,\xi)|.
\]
\end{proof}

\begin{lemma}\label{lem:weightfactorization}
Let $N\in\N$ and $p\in\Delta_{d-1}$. For $1\le j\le d$, set $\lambda_j=Np_j$. Then, for all $\alpha\in\Lambda_{N,d}$,
\begin{equation}
w_{N,p}(\alpha)=f_N(N)^{-1}\prod\limits_{j=1}^df_{\lambda_j}(\alpha_j).\label{eq:weightfactorization}
\end{equation}
\end{lemma}

\begin{proof}
\noindent\textbf{Case 1:}  $p_j>0$ for all $j$.  
Since $\sum\limits_{j=1}^d\lambda_j=N$ and $|\alpha|=N$, we have
\[
f_N(N)^{-1}\prod\limits_{j=1}^df_{\lambda_j}(\alpha_j)=\left(\e^{-\frac N2}\frac{N^{\frac N2}}{\sqrt{N!}}\right)^{-1}\e^{-\frac N2}\frac{N^{\frac N2}p^{\frac\alpha2}}{\sqrt{\alpha!}}=\sqrt{\frac{N!}{\alpha!}}\,p^{\frac\alpha2}=w_{N,p}(\alpha).
\]
\medskip
\noindent\textbf{Case 2:} $p_j=0$ and $\alpha_j>0$ for some $j$.
In this case,  $p^\alpha=0$ and $f_{\lambda_j}(\alpha_j)=f_0(\alpha_j)=0$. Hence both sides of \eqref{eq:weightfactorization} are zero. 
\medskip
\noindent\textbf{Case 3:} $\alpha_j=0$ whenever $p_j=0$.
Since $f_0(0)=1$, $\sum\limits_{p_j>0}p_j=1$ and $\sum\limits_{p_j>0}\alpha_j=N$, we have
\[
\prod\limits_{j=1}^df_{\lambda_j}(\alpha_j)=\prod\limits_{p_j>0}\e^{-\frac{Np_j}{2}}\frac{(Np_j)^{\frac{\alpha_j}{2}}}{\sqrt{\alpha_j!}}=\e^{-\frac N2}\frac{N^{\frac N2}p^{\frac\alpha2}}{\sqrt{\alpha!}}.
\]
Dividing by $f_N(N)$ gives \eqref{eq:weightfactorization}.
\end{proof}

Let $d\ge2$, $N\ge d$ and $p\in\Delta_{d-1}$. For $1\le j\le d$, set $\lambda_j=Np_j$. After relabeling the coordinates, we assume that $\lambda_d=\max\limits_{1\le j\le d}\lambda_j$. Then
\begin{equation}
\lambda_d\ge\frac Nd\ge1.\label{eq:lambdad}
\end{equation}
Let $J=\{i<d:\lambda_i\ge1\}$ and $I=\{i<d:0\le\lambda_i<1\}$. Write $J=\{j_1,\dots,j_r\}$, where $r=|J|$ and $j_1<\dots<j_r$. For $\lambda\ge1$, let $f_\lambda$ denote the smooth extension given by Lemma \ref{lem:fsmooth}. For $k=(k_i)_{i\in I}\in\N_0^I$ and $\xi\in\R^d$, set $K=\sum\limits_{i\in I}k_i$ and define
\begin{equation}
T_k=\sum\limits_{n\in\Z^r}\Big(\prod\limits_{a=1}^rf_{\lambda_{j_a}}(n_a)\Big)f_{\lambda_d}\Big(N-K-\sum\limits_{a=1}^rn_a\Big)\e(\Phi_k(n)),\label{eq:Tk}
\end{equation}
where
\[
\begin{aligned}
\Phi_k(n)=\sqrt2\left(\sum\limits_{a=1}^rn_a^2+\sum\limits_{i\in I}k_i^2+\left(N-K-\sum\limits_{a=1}^rn_a\right)^2\right)+\sum\limits_{a=1}^r\xi_{j_a}n_a+\sum\limits_{i\in I}\xi_ik_i+\xi_d\Big(N-K-\sum\limits_{a=1}^rn_a\Big).
\end{aligned}
\]

\begin{lemma}\label{lem:Tk}
With the notation above,
there is a constant $A_*\ge1$,
independent of $d$, $N$, $p$, $k$, and $\xi$, such that
\begin{equation}\label{eq:Tkbound}
|T_k|\le A_*^{r+1}\lambda_d^{-\frac 14}
\end{equation}
for $N\ge d$, $p\in\Delta_{d-1}$ satisfying
\eqref{eq:lambdad}, $k\in\N_0^I$ and $\xi\in\R^d$.
\end{lemma}

\begin{proof}
	By Proposition~\ref{lem:profiles-intro}, there exists $B_0\ge1$, independent of $\lambda$, such that
	\begin{equation}\label{eq:profile-fixed-order}
		\max_{0\le j\le2}\sup_{\lambda\ge1}\sup_{u\in\R}
		(1+|u|)^2|g_\lambda^{(j)}(u)|\le B_0.
	\end{equation}
	In particular,
	\[
	\|g_\lambda\|_{L^1(\R)},\ \|g_\lambda''\|_{L^1(\R)}
	\le B_0\int_\R(1+|u|)^{-2}\,du=2B_0.
	\]
	Therefore, for \(t\ne0\), by integration by parts, we have
	\[
	|\widehat g_\lambda(t)|
	\le
	\frac{\|g_\lambda''\|_{L^1(\R)}}{4\pi^2t^2}\leq \frac{B_0}{2\pi^2t^2}.
	\]
	Consequently,
	\begin{equation}\label{eq:profile-Fourier-L1}
	\begin{aligned}
		\|\widehat g_\lambda\|_{L^1(\R)}
		&=\int_{|t|\le1}|\widehat g_\lambda(t)|\,dt
		+\int_{|t|>1}|\widehat g_\lambda(t)|\,dt\\
		&\leq \int_{|t|\le1} \|g_\lambda\|_{L^1(\R)}\,dt
		+\int_{|t|>1}|\widehat g_\lambda(t)|\,dt\\
		&\le4B_0+\frac{B_0}{2\pi^2}\int_{|t|>1}t^{-2}\,dt\\
		&=\left(4+\frac1{\pi^2}\right)B_0
		\le5B_0.
	\end{aligned}
	\end{equation}
	
	For $r\ge1$, by the definition of $g_\lambda$,
	\[
	f_{\lambda_{j_a}}(n_a)
	=\lambda_{j_a}^{-\frac14}
	g_{\lambda_{j_a}}
	\left(\frac{n_a-\lambda_{j_a}}{\sqrt{\lambda_{j_a}}}\right),
	\qquad 1\le a\le r.
	\]
	Moreover, by Fourier inversion,
	\begin{align}\label{fd}
	\begin{aligned}
		f_{\lambda_d}\left(N-K-\sum_{a=1}^r n_a\right)
		&=\lambda_d^{-\frac14}
		g_{\lambda_d}\left(
		\frac{N-K-\lambda_d-\sum\limits_{a=1}^r n_a}{\sqrt{\lambda_d}}
		\right)\\
		&=\lambda_d^{-\frac14}
		\int_\R\widehat g_{\lambda_d}(t)
		\e\left(\frac{t(N-K-\lambda_d)}{\sqrt{\lambda_d}}
		-\frac{t}{\sqrt{\lambda_d}}\sum_{a=1}^r n_a\right)\,dt.
	\end{aligned}
	\end{align}
	By the definition of $\Phi_k$,
	\begin{equation}\label{eq:Phi}
		\begin{aligned}
			\Phi_k(n)
			=\sqrt2\,n^TA_{r,1}n
			+\sum_{a=1}^r
			\bigl(\xi_{j_a}-\xi_d-2\sqrt2(N-K)\bigr)n_a + \gamma_k,
		\end{aligned}
	\end{equation}
	where
	\[
	\gamma_k=\sqrt2\left((N-K)^2+\sum_{i\in I}k_i^2\right)
	+\xi_d(N-K)+\sum_{i\in I}\xi_i k_i.
	\]
	Substituting \eqref{eq:Phi} and \eqref{fd} into \eqref{eq:Tk}, we obtain
	\[
	\begin{aligned}
		T_k
		&=\e\left(\gamma_k\right)
		\lambda_d^{-\frac14}
		\left(\prod_{a=1}^r\lambda_{j_a}^{-\frac14}\right)
		\int_\R\widehat g_{\lambda_d}(t)
		\e\left(\frac{t(N-K-\lambda_d)}{\sqrt{\lambda_d}}\right)
		\sum_{n\in\Z^r}\prod_{a=1}^r
		g_{\lambda_{j_a}}
		\left(\frac{n_a-\lambda_{j_a}}{\sqrt{\lambda_{j_a}}}\right)\\
		&\qquad\qquad\times
		\e\left(\sqrt2\,n^TA_{r,1}n+
		\sum_{a=1}^r
		\left(\xi_{j_a}-\xi_d-2\sqrt2(N-K)
		-\frac{t}{\sqrt{\lambda_d}}\right)n_a\right)\,dt.
	\end{aligned}
	\]
	Since $\widehat g_{\lambda_d}\in L^1(\mathbb R)$ and
	$g_{\lambda_{j_a}}\in\mathcal S(\mathbb R)$,
	the interchange of summation and integration is justified
	by absolute convergence.
	By \eqref{eq:profile-fixed-order} and
	Proposition~\ref{prop:tensor-quant}, for every $t\in\R$,
	\[
	\begin{aligned}
		&\left|\sum_{n\in\Z^r}\prod_{a=1}^r
		g_{\lambda_{j_a}}
		\left(\frac{n_a-\lambda_{j_a}}{\sqrt{\lambda_{j_a}}}\right)
		\e\left(\sqrt2\,n^TA_{r,1}n+
		\sum_{a=1}^r
		\left(\xi_{j_a}-\xi_d-2\sqrt2(N-K)
		-\frac{t}{\sqrt{\lambda_d}}\right)n_a\right)\right|\\
		\le&(CB_0)^r\prod_{a=1}^r\lambda_{j_a}^{\frac14},
	\end{aligned}
	\]
	where $C$ is the constant in Proposition~\ref{prop:tensor-quant}
	and is independent of $r$ and $t$.
	Therefore, by \eqref{eq:profile-Fourier-L1},
	\begin{equation*}
		\begin{aligned}
			|T_k|
			\le\lambda_d^{-\frac14}
			\left(\prod_{a=1}^r\lambda_{j_a}^{-\frac14}\right)
			\int_\R|\widehat g_{\lambda_d}(t)|
			(CB_0)^r\prod_{a=1}^r\lambda_{j_a}^{\frac14}\,dt
			\le 5B_0(CB_0)^r\lambda_d^{-\frac14}.
		\end{aligned}
	\end{equation*}
	For $r=0$, by \eqref{eq:Tk} and
	\eqref{eq:profile-fixed-order},
	\[
	|T_k|
	=|f_{\lambda_d}(N-K)|
	=\lambda_d^{-\frac14}
	\left|g_{\lambda_d}
	\left(\frac{N-K-\lambda_d}{\sqrt{\lambda_d}}\right)\right|
	\le B_0\lambda_d^{-\frac14}.
	\]
	Taking $A_*=B_0\max\{5,C\}$ gives \eqref{eq:Tkbound} in both cases.
\end{proof}

\section{Proof of Theorem~\ref{thm:mainK}}\label{sec:proof}
We first construct the polynomials $p_{N,j}$ in the Hardy space.
For $d\geq2$ and $N\geq0$, set
\[
I_{N+d-1,d-1}=\Big\{(\beta_1,\dots,\beta_{d-1})\in\Z^{d-1}:0\le\beta_1<\dots<\beta_{d-1}\le N+d-2\Big\}.
\]
\begin{lemma}\label{lem:index-bijection}
	For $d\geq2$ and $N\geq0$, the map
	$\beta:\Lambda_{N,d}\to I_{N+d-1,d-1}$ defined by
	\begin{equation*}
		\beta_j(\alpha)=\sum\limits_{k=1}^j\alpha_k+j-1,
		\ 1\leq j\leq d-1
	\end{equation*}
	is a bijection.
\end{lemma}

\begin{proof}
	For $\alpha\in\Lambda_{N,d}$, we have
	\[
	\beta_1(\alpha)=\alpha_1\geq0,\qquad
	\beta_j(\alpha)-\beta_{j-1}(\alpha)=\alpha_j+1\geq1,
	\  2\leq j\leq d-1,
	\]
	and
	\[
	\beta_{d-1}(\alpha)=N-\alpha_d+d-2\leq N+d-2.
	\]
	Thus $\beta(\alpha)\in I_{N+d-1,d-1}$.
	Conversely, for $\beta\in I_{N+d-1,d-1}$, define
	\begin{equation}\label{eq:alpha}
		\alpha_j=
		\begin{cases}
			\beta_1,&j=1,\\
			\beta_j-\beta_{j-1}-1,&2\leq j\leq d-1,\\
			N+d-2-\beta_{d-1},&j=d.
		\end{cases}
	\end{equation}
	Then $\alpha\in\Lambda_{N,d}$ and for $1\leq j\leq d-1$, \(\sum\limits_{k=1}^j\alpha_k + j-1=\beta_j\).
	Thus \eqref{eq:alpha} gives a preimage of every $\beta$.
	To prove injectivity, suppose that $\beta(\alpha)=\beta(\alpha')$.
	Applying \eqref{eq:alpha} to both sides yields
	$\alpha=\alpha'$.
\end{proof}

 For each $\rho\in I_{N+d-1,d-1}$, define
\begin{equation}
	p_{N,\rho}(z)=\frac{\sqrt{D_{N,d}}}{(N+d-1)^{\frac{d-1}{2}}}\sum\limits_{\alpha\in\Lambda_{N,d}}\det\left[\e\left(\frac{\rho_i\left(\sum\limits_{k=1}^j\alpha_k+j-1\right)}{N+d-1}\right)\right]_{i,j=1}^{d-1}\e\left(\sqrt2\sum\limits_{k=1}^d\alpha_k^2\right)\sqrt{\frac{N!}{\alpha!}}\,z^\alpha.\label{eq:Pdefinition}
\end{equation}
Since $|I_{N+d-1,d-1}|=D_{N,d}$, we enumerate these polynomials as
$\{p_{N,j}:1\le j\le D_{N,d}\}$. Define
\begin{align}\label{FNd}
	F_{N,d}(z)=\sum\limits_{\alpha\in\Lambda_{N,d}}\e\left(\sqrt2\sum\limits_{j=1}^d\alpha_j^2\right)\sqrt{\frac{N!}{\alpha!}}\,z^\alpha.
\end{align}

\begin{lemma}\label{lem:orthonormal}
	The polynomials $\{p_{N,\rho}\}_{\rho\in I_{N+d-1,d-1}}$ defined in \eqref{eq:Pdefinition} form an orthonormal basis of $\PP_N^{(d)}$. Moreover, for \(\rho\in I_{N+d-1,d-1}\),
	\[
	\|p_{N,\rho}\|_\infty\le\sqrt{(d-1)!}\,\|F_{N,d}\|_\infty.
	\]
\end{lemma}

\begin{proof}
	Let $\mathbf F$ be the $(N+d-1)\times(N+d-1)$ matrix defined by
	\[
	\mathbf F_{a,b}=(N+d-1)^{-\frac 12}\e\left(\frac{ab}{N+d-1}\right),\ 0\le a,b\le N+d-2.
	\]
	It is easy to verify that $\mathbf F$ is unitary. For $\rho,\beta\in I_{N+d-1,d-1}$, define
	\begin{equation}
		U_{\rho,\beta}=\det[\mathbf F_{\rho_i,\beta_j}]_{i,j=1}^{d-1}=(N+d-1)^{-\frac{d-1}{2}}\det\left[\e\left(\frac{\rho_i\beta_j}{N+d-1}\right)\right]_{i,j=1}^{d-1}.\label{eq:Udef}
	\end{equation}
	By the Cauchy--Binet formula \cite[p.~28]{HornJohnson2013},
	\begin{equation}
		\sum\limits_{\beta\in I_{N+d-1,d-1}}U_{\rho,\beta}\overline{U_{\rho',\beta}}=\det[(\mathbf F\mathbf F^*)_{\rho_i,\rho'_j}]_{i,j=1}^{d-1}=\det[\delta_{\rho_i\rho'_j}]_{i,j=1}^{d-1}=\delta_{\rho\rho'}.\label{eq:CauchyBinet}
	\end{equation}
	Set
	\[
	E_\alpha(z)=\sqrt{D_{N,d}}\sqrt{\frac{N!}{\alpha!}}\,z^\alpha,\qquad \omega_\alpha=\e\Big(\sqrt2\sum\limits_{j=1}^d\alpha_j^2\Big).
	\]
	By \eqref{eq:Pdefinition} and \eqref{eq:Udef},
	\begin{equation}
		p_{N,\rho}(z)=\sum\limits_{\alpha\in\Lambda_{N,d}}U_{\rho,\beta(\alpha)}\omega_\alpha E_\alpha(z).\label{eq:Pmonomial}
	\end{equation}
	Since the monomials $E_\alpha$ are orthonormal and $|\omega_\alpha|=1$, by \eqref{eq:CauchyBinet}, we have
	\[
	\langle p_{N,\rho},p_{N,\rho'}\rangle=\sum\limits_{\alpha\in\Lambda_{N,d}}U_{\rho,\beta(\alpha)}\overline{U_{\rho',\beta(\alpha)}}\overset{(*)}{=}\sum\limits_{\beta\in I_{N+d-1,d-1}}U_{\rho,\beta}\overline{U_{\rho',\beta}}=\delta_{\rho\rho'},
	\]
	where \((*)\) follows from the bijectivity of $\alpha\mapsto\beta(\alpha)$. Moreover,
	\[
	|I_{N+d-1,d-1}|=\binom{N+d-1}{d-1}=\dim\PP_N^{(d)}.
	\]
	Therefore, $\{p_{N,\rho}\}_{\rho\in I_{N+d-1,d-1}}$ form an orthonormal basis of $\PP_N^{(d)}$.
	
	It remains to estimate $\|p_{N,\rho}\|_\infty$. Let $\mathfrak  S_{d-1}$ denote the set of permutations of $\{1,\dots,d-1\}$. For $\pi\in \mathfrak S_{d-1}$, set
	\[
	c_{\rho,\pi}=\frac{1}{N+d-1}\sum\limits_{i=1}^{d-1}\rho_i(\pi(i)-1),
	\]
	and
	\[
	t_{\rho,\pi,k}=\frac{1}{N+d-1}\sum\limits_{\substack{1\le i\le d-1\\\pi(i)\ge k}}\rho_i,\ 1\le k\le d-1.
	\]
	Then, by Lemma~\ref{lem:index-bijection},
	\[
	\frac{1}{N+d-1}\sum\limits_{i=1}^{d-1}\rho_i\beta_{\pi(i)}(\alpha)=\frac{1}{N+d-1}\sum\limits_{i=1}^{d-1}\rho_i\left(\sum\limits_{k=1}^{\pi(i)}\alpha_k+\pi(i)-1\right)=c_{\rho,\pi}+\sum\limits_{k=1}^{d-1}t_{\rho,\pi,k}\alpha_k.
	\]
	Therefore, by \eqref{eq:Udef}, we have
	\begin{equation}
		\begin{aligned}
			U_{\rho,\beta(\alpha)}&=(N+d-1)^{-\frac {d-1}{2}}\sum\limits_{\pi\in\mathfrak  S_{d-1}}\operatorname{sgn}(\pi)\e\left(\frac{1}{N+d-1}\sum\limits_{i=1}^{d-1}\rho_i\beta_{\pi(i)}(\alpha)\right)\\
			&=(N+d-1)^{-\frac {d-1}{2}}\sum\limits_{\pi\in\mathfrak  S_{d-1}}\operatorname{sgn}(\pi)\e\left(c_{\rho,\pi}+\sum\limits_{k=1}^{d-1}t_{\rho,\pi,k}\alpha_k\right).
		\end{aligned}\label{eq:Uexpanded}
	\end{equation}
	We insert \eqref{eq:Uexpanded} and the expression for $E_\alpha$
	into \eqref{eq:Pmonomial} and use \eqref{FNd} to obtain
	\begin{equation*}
		\begin{aligned}
			\|p_{N,\rho}\|_\infty
			&=\frac{\sqrt{D_{N,d}}}{(N+d-1)^{\frac {d-1}{2}}}
			\sup\limits_{z\in \partial\mathbb B_d}\left|
			\sum\limits_{\pi\in\mathfrak  S_{d-1}}\operatorname{sgn}(\pi)
			\e(c_{\rho,\pi})
			F_{N,d}\big(\e(t_{\rho,\pi,1})z_1,\dots,
			\e(t_{\rho,\pi,d-1})z_{d-1},z_d\big)\right|\\
			&\leq\frac{\sqrt{D_{N,d}}\,(d-1)!}
			{(N+d-1)^{\frac {d-1}{2}}}\|F_{N,d}\|_\infty\\
			&\leq\sqrt{(d-1)!}\,\|F_{N,d}\|_\infty,
		\end{aligned}
	\end{equation*}
	where the last inequality follows from
	\[
	\binom{N+d-1}{d-1}
	=\frac{(N+d-1)(N+d-2)\cdots(N+1)}{(d-1)!}
	\leq\frac{(N+d-1)^{d-1}}{(d-1)!}.
	\]
\end{proof}
	By Stirling's formula \cite{Olver1997}, there exists a constant
	$C_{\mathrm{St}}>0$ such that \(N!\le C_{\mathrm{St}}\sqrt N\left(\frac Ne\right)^N\), \(N\ge1\).
	Hence
	\begin{align}\label{fnxia}
	f_N(N)=\left(\frac{\e^{-N}N^N}{N!}\right)^{\frac 12}
	\ge C_{\mathrm{St}}^{-\frac 12}N^{-\frac14}.
	\end{align}
\begin{proposition}\label{prop:seed}
Let $A_*$ be the constant in Lemma~\ref{lem:Tk}. 
Set
\[
A=\max\limits\left\{2,\max\limits\left\{1,\sqrt{C_{\mathrm{St}}}\right\}
\max\limits\left\{A_*,\sum\limits_{m=0}^{\infty}\frac1{\sqrt{m!}}\right\}\right\}.
\]
Then $A$ is independent of $d$ and $N$, and for every $d,N\in\N$,
\[
\|F_{N,d}\|_\infty\le A^d d^{\frac 14}.
\]
\end{proposition}

\begin{proof}
If $d=1$, then $F_{N,1}(z)=\e(\sqrt2\,N^2)z^N$, and hence $\|F_{N,1}\|_\infty=1$. 
Fix $p\in\Delta_{d-1}$ and $\xi\in\mathbb R^d$.
For $1\le N<d$, by the Cauchy--Schwarz inequality and \eqref{eq:multinomial},
\begin{equation}
|S_N(p,\xi)|\le D_{N,d}^{\frac12}\left(\sum\limits_{\alpha\in\Lambda_{N,d}}w_{N,p}(\alpha)^2\right)^{\frac12}=\sqrt{D_{N,d}}\le\sqrt{\binom{2d-2}{d-1}} \le2^{d-1}.\label{eq:lowdegree}
\end{equation}
Suppose that $N\geq d$.
Set $\lambda_j=Np_j$ for $1\leq j\leq d$.
After relabeling the coordinates, assume that
$\lambda_d=\max\limits_{1\leq j\leq d}\lambda_j$,
and use the notation introduced before Lemma~\ref{lem:Tk}. 
For each $k\in\N_0^I$, let
\[
\Lambda_{N,d}(k)
=\{\alpha\in\Lambda_{N,d}:\alpha_i=k_i\text{ for }i\in I\},
\qquad
\Delta_{N-K}^{(r)}
=\left\{n\in\N_0^r:\sum_{a=1}^r n_a\le N-K\right\}.
\]
Since $N=|\alpha|=\sum\limits_{i\in I}\alpha_i+\sum\limits_{a=1}^r\alpha_{j_a}+\alpha_d$, the map $\alpha\mapsto(\alpha_{j_1},\dots,\alpha_{j_r})$ is a bijection from $\Lambda_{N,d}(k)$ onto $\Delta_{N-K}^{(r)}$, with inverse given by
\begin{equation}\label{eq:inverse}
	\alpha_i=
	\begin{cases}
		k_i, & i\in I,\\
		n_a, & i=j_a,\ 1\le a\le r,\\
		N-K-\displaystyle\sum_{a=1}^r n_a, & i=d.
	\end{cases}
\end{equation}
Since $\lambda_{j_a}\ge1$ and $\lambda_d\ge1$, by Lemma \ref{lem:fsmooth} and \eqref{eq:inverse}, we have
\begin{equation}
\begin{aligned}
&T_k\\
=&\sum\limits_{n \in \Delta_{N-K}^{(r)}}\left(\prod\limits_{a=1}^rf_{\lambda_{j_a}}(n_a)\Big)f_{\lambda_d}\Big(N-K-\sum\limits_{a=1}^rn_a\right)\e(\Phi_k(n))\\
=&\sum\limits_{\alpha \in \Lambda_{N,d}(k)}\left(\prod\limits_{a=1}^rf_{\lambda_{j_a}}(\alpha_{j_a})\right)f_{\lambda_d}(\alpha_d) \e\left(\sqrt2\left[\sum\limits_{a=1}^r\alpha_{j_a}^2+\sum\limits_{i\in I}\alpha_i^2+\alpha_d^2\right]+\sum\limits_{a=1}^r\xi_{j_a}\alpha_{j_a}+\sum\limits_{i\in I}\xi_i\alpha_i+\xi_d\alpha_d\right)\\
=&\sum\limits_{\alpha \in \Lambda_{N,d}(k)}\left(\prod\limits_{j\notin I}f_{\lambda_j}(\alpha_j)\right)\e\left(\sqrt2\sum\limits_{j=1}^d\alpha_j^2+\xi\cdot\alpha\right).
\end{aligned}\label{eq:Tkrewritten}
\end{equation}
By \eqref{eq:SN}, Lemma \ref{lem:weightfactorization},  \(\Lambda_{N,d}=\bigsqcup\limits_{k\in\N_0^I}\Lambda_{N,d}(k)\) and \eqref{eq:Tkrewritten}, we have
\begin{equation}\label{eq:SNdecomposition}
	\begin{aligned}
		S_N(p,\xi)
		&=f_N(N)^{-1}\sum\limits_{\alpha\in\Lambda_{N,d}}
		\left(\prod\limits_{j=1}^df_{\lambda_j}(\alpha_j)\right)
		\e\left(\sqrt2\sum\limits_{j=1}^d\alpha_j^2+\xi\cdot\alpha\right)\\
		&=f_N(N)^{-1}\sum\limits_{k\in\N_0^I}
		\left(\prod\limits_{i\in I}f_{\lambda_i}(k_i)\right)
		\sum\limits_{\alpha \in \Lambda_{N,d}(k)}
		\left(\prod\limits_{j\notin I}f_{\lambda_j}(\alpha_j)\right)
		\e\left(\sqrt2\sum\limits_{j=1}^d\alpha_j^2+\xi\cdot\alpha\right)\\
		&=f_N(N)^{-1}\sum\limits_{k\in\N_0^I}
		\left(\prod\limits_{i\in I}f_{\lambda_i}(k_i)\right)T_k.
	\end{aligned}
\end{equation}
Moreover, when $K>N$,
$\Lambda_{N,d}(k) = \emptyset$, and hence $T_k=0$.
Thus the sum in \eqref{eq:SNdecomposition} is finite. 
Therefore, by \eqref{fnxia} and Lemma~\ref{lem:Tk},
\begin{equation}\label{eq:SNbound}
\begin{aligned}
|S_N(p,\xi)|
&\le \sqrt{C_{\mathrm{St}}}N^{\frac14}A_*^{r+1}\lambda_d^{-\frac14}
\prod\limits_{i\in I}\left(\sum\limits_{k_i=0}^\infty f_{\lambda_i}(k_i)\right)\\
& \le \sqrt{C_{\mathrm{St}}}N^{\frac14}A_*^{r+1}\lambda_d^{-\frac14}
\prod_{i\in I}\left(
\e^{-\frac{\lambda_i}{2}}
\sum_{k_i=0}^{\infty}\frac{\lambda_i^{\frac{k_i}{2}}}{\sqrt{k_i!}}
\right)\\
&\le \sqrt{C_{\mathrm{St}}}A_*^{r+1}\left(\frac{N}{\lambda_d}\right)^{\frac14}
\left(\sum_{k=0}^{\infty}\frac1{\sqrt{k!}}\right)^{|I|}\\
&\overset{(*)}{\le} \sqrt{C_{\mathrm{St}}}\max\limits\left\{A_*,\sum\limits_{m=0}^{\infty}\frac1{\sqrt{m!}}\right\}^{d}d^{\frac14}\leq A^d d^{\frac14},
\end{aligned}
\end{equation}
where \((*)\) follows from  $r+1+|I|=d$ and \eqref{eq:lambdad}.
Combining \eqref{eq:SNbound} with \eqref{eq:lowdegree}, by Lemma~\ref{lem:Fsup}, the result holds.
\end{proof}

\begin{proof}[Proof of Theorem~\ref{thm:mainK}]
		Let $A\ge1$ be the constant in
		Proposition~\ref{prop:seed}.
		
		\noindent\textit{\((i)\Rightarrow(ii)\).}
		For $d\ge2$ and $m\ge1$, Lemma~\ref{lem:orthonormal} and
		Proposition~\ref{prop:seed} give an orthonormal basis
		$\{p_{m,k}:1\le k\le D_{m,d}\}$ of $\mathcal P_m^{(d)}$
		with respect to the $H^2(\partial\mathbb B_d)$ inner product.
		For $d=1$, take $p_{m,1}(z)=z^m$, and in every dimension
		take $p_{0,1}=1$. Then
		\begin{align}\label{hardy}
		\sup_{m\ge0}\sup_{1\le k\le D_{m,d}}
		\|p_{m,k}\|_\infty
		\le A^d d^{\frac14}\sqrt{(d-1)!}.
		\end{align}
		Define
		$e_{m,k}(z)=\sqrt{\frac{a_m}{D_{m,d}}}\,p_{m,k}(z)$.
		Then
		\[
		\begin{aligned}
			\langle e_{m,i},e_{m,j}\rangle_H
			=\frac{a_m}{D_{m,d}}
			\langle p_{m,i},p_{m,j}\rangle_H
			=\frac{a_m}{D_{m,d}}\cdot\frac{D_{m,d}}{a_m}
			\langle p_{m,i},p_{m,j}\rangle_{H^2(\partial\mathbb B_d)}
			=\delta_{ij}.
		\end{aligned}
		\]
		Since \(\dim\mathcal P_m^{(d)}=D_{m,d}\), the family
		$\{e_{m,k}:1\leq k\leq D_{m,d}\}$ is an orthonormal basis of
		$\mathcal P_m^{(d)}$. Moreover,
		\[
		H=\overline{\bigoplus_{m\geq0}\mathcal P_m^{(d)}}^{\,\|\cdot\|_H}.
		\]
		Therefore, $\{e_{m,k}:m\geq0,\ 1\leq k\leq D_{m,d}\}$ is an
		orthonormal basis of $H$ consisting of homogeneous holomorphic
		polynomials.
		Furthermore, by \eqref{hardy},
		\[
		\begin{aligned}
			\sup_{m\ge0}\sup_{1\le k\le D_{m,d}}\|e_{m,k}\|_\infty
			&\le
			\left(\sup_{m\ge0}\frac{a_m}{D_{m,d}}\right)^{\frac12}
			\sup_{m\ge0}\sup_{1\le k\le D_{m,d}}\|p_{m,k}\|_\infty\\
			&\le A^d d^{\frac14}\sqrt{(d-1)!}
			\left(\sup_{m\ge0}\frac{a_m}{D_{m,d}}\right)^{\frac12}\\
			&\leq \exp\left(\frac12d\log d+O(d)\right)\left(\sup_{m\ge0}\frac{a_m}{D_{m,d}}\right)^{\frac12} .
		\end{aligned}
		\]
		
		\noindent\textit{\((ii)\Rightarrow(i)\).}
		Suppose that $H$ admits an orthonormal basis
		$\{q_n\}_{n=1}^\infty$ consisting of homogeneous holomorphic
		polynomials such that $\sup\limits_{n\ge1}\|q_n\|_\infty<\infty$.
		For each $m\ge0$, there exists a basis vector
		$q_{n_m}\in\mathcal P_m^{(d)}$.
		Otherwise, $\mathcal P_m^{(d)}\perp H$, a contradiction.
		Since $\|q_{n_m}\|_H=1$, we have
		\[
		\frac{a_m}{D_{m,d}}
		=\frac{a_m}{D_{m,d}}\|q_{n_m}\|_H^2
		=\|q_{n_m}\|_{H^2(\partial\mathbb B_d)}^2
		\le\|q_{n_m}\|_\infty^2
		\le\sup_{n\ge1}\|q_n\|_\infty^2
		<\infty.
		\]
		Taking the supremum over $m\ge0$ proves \textup{(i)}.
\end{proof}

\noindent{\bf Declaration of AI assistance.} Dr.~Zeyou Zhu came across the problem of the existence of inner functions while reading Rudin's
\emph{Function Theory in the Unit Ball of $\mathbb C^n$}.
He then found Bourgain's work on uniformly bounded
orthonormal bases, which can be used to construct
inner functions.
Following this, we began studying Rudin's homogeneous
basis problem about two years ago.
The problem later became the topic of Dr.~Fan Zhang's
doctoral dissertation. At first, we only obtained the existence of uniformly bounded orthonormal bases in unitarily invariant reproducing kernel Hilbert spaces for $d=2,3$. 
At the beginning of 2026, we designed the outline of our study on this problem. 
With the help of ChatGPT (OpenAI), we optimized the research methods for this problem in May and got the solution in August 2026. 
During the preparation of this manuscript, ChatGPT (OpenAI) assisted with checking intermediate computations and estimates, revising the language, and preparing the LaTeX source.
The authors take responsibility for the mathematical statements, proofs, and final text.

\noindent{\bf Additional Statement on This Paper.} After completing our manuscript, we became aware of the work
of Gon\c{c}alves, Ferreira and Radziwi{\l}{\l}
\cite{GoncalvesFerreiraRadziwill2026}, which had recently
been posted on arXiv and independently resolves Rudin's problem.
After comparing our work with that of Gon\c{c}alves, Ferreira
	and Radziwi{\l}{\l}, we find that our approach and basis
	construction differ from theirs.
Furthermore, we obtain an explicit bound for the supremum norms
of the basis elements in terms of the dimension.
More precisely, as proved in Corollary~\ref{cor:rudin},
\[
\|p_{N,j}\|_\infty
\leq \exp\left(\frac12 d\log d+O(d)\right),
\qquad 1\leq j\leq D_{N,d},
\]
where the implicit constant in $O(d)$ is independent of $d$, $N$, and $j$.
Gon\c{c}alves, Ferreira and Radziwi{\l}{\l}
	\cite[p.~2, first comment after Theorem~1]
	{GoncalvesFerreiraRadziwill2026}
	also discuss the dependence of their bound on the dimension.
	They note that an explicit constant can be obtained from
	their proof and expect an upper bound of the form
	$\exp(\exp(Cd\log d))$, where $C>0$ is a constant.

\bigskip

\noindent Penghui Wang\\
School of Mathematics, Shandong University,
Jinan, 250100, China\\
E-mail: \texttt{phwang@sdu.edu.cn}

\medskip

\noindent Fan Zhang\\
School of Mathematics, Shandong University,
Jinan, 250100, China\\
E-mail: \texttt{202211823@sdu.edu.cn}

\medskip

\noindent Zeyou Zhu\\
School of Mathematical Sciences, Fudan University,
Shanghai, 200433, China\\
E-mail: \texttt{zhuzeyou@fudan.edu.cn}

\end{document}